\documentclass[11pt]{amsart}  
\usepackage[left=2.5cm,right=2.5cm,top=2.5cm,bottom=2.5cm]{geometry}
\usepackage{graphicx}
\usepackage{multirow}
\usepackage{hyperref}
\usepackage{bbm}
\usepackage{amsmath,amssymb,amsfonts, amsthm,epsfig}
\usepackage{verbatim,color}
\usepackage{enumerate}
\usepackage{mathtools}

\numberwithin{equation}{section}

\newtheorem{theorem}{Theorem}[section]
\newtheorem{lemma}[theorem]{Lemma}

\newtheorem{remark}[theorem]{Remark}

\newenvironment{Proof}{\removelastskip\par\medskip
\noindent{\em Proof.} \rm}{\penalty-20\null\hfill$\square$\par\medbreak}

\numberwithin{equation}{section}
\newcommand{\Ai}{\text{Ai\,}}

\newcommand{\Tr}{\text{\rm{tr}\,}}

\newcommand{\bK}{\mathbf{K}}
\newcommand{\bA}{\mathbf{A}}
\newcommand{\bI}{\mathbf{I}}
\newcommand{\bQ}{\mathbf{Q}}
\newcommand{\bP}{\mathbf{P}}
\newcommand{\bM}{\mathbf{M}}
\newcommand{\bk}{\mathbf{k}}
\newcommand{\bR}{\mathbf{R}}
\newcommand{\bL}{\mathbf{L}}
\newcommand{\bJ}{\mathbf{J}}
\newcommand{\bbR}{\mathbb{R}}
\newcommand{\pii}{2\pi\mathrm{i}}

\begin{document}
\setcounter{page}{1}


\title[Long and short time asymptotics]
{Long and short time asymptotics of the two-time distribution in local random growth}
\author[K.~Johansson]{Kurt Johansson}

\thanks{Supported by the grant KAW 2015.0270 from the Knut and Alice Wallenberg Foundation and grant 2015-04872 from the Swedish Science Research Council (VR)}

\address{
Department of Mathematics,
KTH Royal Institute of Technology,
SE-100 44 Stockholm, Sweden}

\email{kurtj@kth.se}

\begin{abstract}
The two-time distribution gives the limiting joint distribution of the heights at two different times of a local 1D random growth model in the curved geometry.
This distribution has been computed in a specific model but is expected to be universal in the KPZ universality class. Its marginals are the GUE Tracy-Widom distribution.
In this paper we study two limits of the two-time distribution. The first, is the limit of long time separation when the quotient of the two times goes to infinity, and the second, is the 
short time limit when the quotient goes to zero.

\end{abstract}

\maketitle

\section{Introduction and results}\label{sect1}

\subsection{Introduction}\label{subsectIntro}
In this paper we will consider the short and long time separation limits of the asymptotic two-time distribution in a polynuclear growth model or, equivalently in a directed last-passage percolation model. For background on these models which belong to the so called KPZ universality class, we refer to \cite{Corw} and \cite{Quas}.
Let us recall the result on the two-time distribution from \cite{Jott}.
Let $\left(w(i,j)\right)_{i,j\ge 1}$ be independent geometric random variables with parameter $q$,
$$
\mathbb{P}[w(i,j)=k]=(1-q)q^k,\quad k\ge 0.
$$
Consider the last-passage times
\begin{equation}\label{gmn}
G(m,n)=\max_{\pi:(1,1)\nearrow (m,n)} \sum_{(i,j)\in\pi} w(i,j),
\end{equation}
where the maximum is over all up/right paths from $(1,1)$ to $(m,n)$, see \cite{JoSh}. It follows from (\ref{gmn}) that we have stochastic recursion formula
\begin{equation}\label{growth}
G(m,n)=\max(G(m-1,n), G(m,n-1))+w(m,n).
\end{equation}
We can relate this to a random growth model with an evolving height function as follows. Let $G(m,n)=0$ if $(m,n)\notin\mathbb{Z}_+^2$, and define the height function $h(x,t)$ by
\begin{equation}\label{hxt}
h(x,t)=G\left(\frac{t+x+1}2,\frac{t-x+1}2\right)
\end{equation}
for $x+t$ odd, and extend it to all $x\in\mathbb{R}$ by linear interpolation. Then (\ref{growth}) leads to a growth rule for $h(x,t)$ and this is the discrete time and space polynuclear growth model. We think of $x\mapsto h(x,t)$ as the height above $x$ at time $t$, and we get a random one-dimensional interface. Let the constants $c_i$ be given by 
\begin{equation}\label{scalingconstants}
c_1=q^{-1/6}(1+\sqrt{q})^{2/3},\quad
c_2=\frac{2\sqrt{q}}{1-\sqrt{q}},\quad c_3=\frac{q^{1/6}(1+\sqrt{q})^{1/3}}{1-\sqrt{q}}.
\end{equation}
Consider the rescaled height function
\begin{equation}\label{hrescaled}
\mathcal{H}_T(\eta,t)=\frac{h(2c_1\eta(tT)^{2/3},2tT)-c_2tT}{c_3(tT)^{1/3}},
\end{equation}
as a process in $\eta\in\mathbb{R}$ and $t>0$. It follows from \cite{JoDPG} that for a fixed $t>0$, the process converges, as $T\to\infty$, to $\mathcal{A}_2(\eta)-\eta^2$, where $\mathcal{A}_2(\eta)$ is the Airy-2-process. In particular,
for any fixed $\eta,t$, 
\begin{equation*}
\lim_{T\to\infty}\mathbb{P}[\mathcal{H}_T(\eta,t)\le\xi-\eta^2]=F_2(\xi)=\det(I-K_{\Ai})_{L^2(\xi,\infty)},
\end {equation*}
where $F_2$ is the Tracy-Widom distribution, and
\begin{equation}\label{Airykernel}
K_{\text{Ai}}(x,y)=\int_0^\infty \Ai(x+s)\Ai(y+s)\,ds,
\end{equation}
is the Airy kernel. 

It is proved in \cite{Jott} that we have the following limit theorem for the joint distribution of two height functions at different times. For $0<t_1<t_2$, and any fixed real numbers
$\eta_1,\eta_2,\xi_1$ and $\xi_2$,
\begin{equation}\label{ttlimit}
\lim_{T\to\infty}\mathbb{P}[\mathcal{H}_T(\eta_1,t_1)\le\xi_1,\mathcal{H}_T(\eta_2,t_2)\le\xi_2]=F_{tt}(\xi_1,\eta_1;\xi_2,\eta_2;\alpha),
\end{equation}
where
\begin{equation}\label{alpha}
\alpha=\frac{t_1}{t_2-t_1},
\end{equation}
and $F_{tt}$ is the two-time distribution function. This two-time distribution is expected to be universal in the KPZ universality class and should be the two-time limit
in many models. For example it is also the limit in Brownian directed last-passage percolation, see \cite{JoTt}.
In \cite{Jott}, different formulas for $F_{tt}$ are given. We will use the one given in (\ref{Fttformula1}) below. If we let
$h(t,\eta)$ be the limiting height function at rescaled time $t$ and rescaled position $\eta$, i.e. the limit of the random variable $\mathcal{H}_T(\eta,t)$ as $T\to\infty$, then
\begin{equation}\label{ttheight}
F_{tt}(\xi_1,\eta_1;\xi_2,\eta_2;\alpha)=\mathbb{P}[h(t_1,\eta_1)\le\xi_1,h(t_2,\eta_2)\le\xi_2].
\end{equation}
We are interested in the limits $\alpha\to 0$ (\it long time limit\rm) and $\alpha\to\infty$ (\it short time limit\rm) of the two-time distribution. The first limit corresponds to the limit of large time
separation, and the second limit to the case of short time separation.

We will show that, as $\alpha\to 0$, we have the asymptotic formula,
\begin{equation}\label{FttIntro}
F_{tt}(\xi_1,\eta_1;\xi_2,\eta_2;\alpha)=F_2(\xi_1+\eta_1^2)F_2(\xi_2+\eta_2^2)(1+e_1\alpha+e_2\alpha^2)+O(\alpha^3),
\end{equation}
where $e_1$ and $e_2$ are explicit but complicated functions of $\xi_1,\eta_1,\xi_2,\eta_2$, see Theorem \ref{Thmalphazero} below. Write $h_1=h(t_1,\xi_1)$ and $h_2=h(t_2,\xi_2)$.
For the limit $\alpha\to\infty$, we will show that, if we define
the rescaled height increment by
\begin{equation*}
h=(1+\alpha^3)^{1/3}h_2-\alpha h_1,
\end{equation*}
then
\begin{equation}\label{FttIntro2}
\lim_{\alpha\to\infty}\frac{\partial}{\partial\xi_1}\mathbb{P}[h_1\le\xi_1,h\le\xi]
=F_2'(\xi_1+\eta_1^2)\frac{\partial}{\partial\xi}\big(F_2(\xi+\eta^2)\psi(\xi,\eta)\big),
\end{equation}
with an explicit function $\psi(\xi,\eta)$, see Theorem \ref{Thmalphainfty}. Here $\eta_1$, $\xi_1$, $\eta$ and $\xi$ are fixed and related to $\eta_2$, $\xi_2$ by (\ref{xi2eta2}) below.
If
$\eta=0$, then $\frac{\partial}{\partial\xi}(F_2(\xi)\psi(\xi,0))$ in the right side of (\ref{FttIntro2}) is the Baik-Rains distribution $F_0(\xi)$, see e.g. \cite{FeSp}. Intuitively, we can understand this from the fact that we can think of $h$ as the evolved height starting from a stationary initial condition. Recall that the Airy process at time $t_1$ is locally Brownian,
\cite{Ha}, \cite{CoHa}. For a more precise version of this heuristic argument, the ergodicity of the KPZ fixed point, see \cite{Pim}.
The limit of $F_{tt}(\xi_1,\eta_1;\xi_2,\eta_2;\alpha)$ with fixed $\eta_1,\eta_2,\xi_1,\xi_2$ as $\alpha\to\infty$ can also be investigated and leads to the two-point
distribution in the Pr\"ahofer-Spohn form. This result is not immediate and we will not discuss it here.

These kinds of results have also been derived non-rigorously using the replica method 
by de Nardis and Le Doussal, \cite{NarDou}, in the limit $t_1/t_2\to 1$, i.e. $\alpha\to\infty$, and by
Le Doussal in \cite{Dous} for the limit $t_1/t_2\to 0$, which means $\alpha\to 0$, to give conjecturally exact formulas. We have checked that the
$\alpha$-coefficient $e_1$ in (\ref{FttIntro}) for $\eta_1=\eta_2=0$, see (\ref{e1spec}), agrees with the result in \cite{Dous}. 
It remains to check that $e_2$ also agrees, and that the formulas agree when $\eta_1,\eta_2\neq 0$. Another comparison between results obtained using the replica method and the exact formula, in the limit when $h(t_1,\eta_1)$ is large,
is given in \cite{dNLD}.
The result in \cite{NarDou} agrees with (\ref{FttIntro2}). 
Continuing earlier work in \cite{FerSpo}, the long and short time limits are also analyzed rigorously in \cite{FerOcc} using a variational approach. The object investigated in these
papers is not the two-time distribution but rather the two-time correlation function of the the heights. See also \cite{BG} for another paper on two-time correlations.
Furthermore, there are very interesting experimental and numerical results on the two-time problem by K. A. Takeuchi and collaborators, see \cite{Take}, \cite{TaSa} and \cite{NaDoTa}.

The two-time problem, and more generally the multi-time problem, have also been studied recently in another, related model by J. Baik and Z. Liu, \cite{BaiLiu}. Very recently, the multi-time, and a more general problem where the starting points in the last-passage percolation problem can also be varied, called the directed landscape, has been investigated in the paper \cite{DOV}. This is closely related to the recent work on the KPZ fixed point, see \cite{MaQuRe}. The multi-time problem in the present setting is studied in \cite{JR}, \cite{Liu}.

\bigskip
\noindent
{\bf Notation}. Throughout the paper $1(\cdot)$ denotes an indicator function, $\gamma_r(a)$  is a positively oriented circle of radius $r$ around the point $a$, and $\gamma_r=\gamma_r(0)$. Also, $\Gamma_c$ is the upward oriented straight line through the point $c$, $t\mapsto c+it$, $t\in\mathbb{R}$. 

\bigskip
\noindent
{\bf Acknowledgement}. I thank Pierre Le Doussal and Jacopo de Nardis for helpful discussions and correspondence.

\subsection{Results for the long time limit}\label{subsectLong}

Before we can state our results we introduce some notation. Given a sufficiently regular function $A(v)$ on $\mathbb{R}_+$, we write
\begin{equation}\label{Ak}
A^{(k)}(v)=(-1)^{k}\frac{d^kA}{dv^k},
\end{equation}
\begin{equation}\label{Akminus}
A^{(-k)}(v)=\int_0^{\infty}\frac{\lambda^{k-1}}{(k-1)!}A(v+\lambda)\,d\lambda,
\end{equation}
for $k\ge 1$, and $A^{(0)}(v)=A(v)$. To the function $A$ we associate the operator $\mathbf{A}$ on $L^2(\mathbb{R}_+)$ with kernel
\begin{equation}\label{Akernel}
\bA(v_1,v_2)=A(v_1+v_2).
\end{equation}
For two functions $A,B$ on $\mathbb{R}_+$, $A\otimes B$ is the rank one operator with kernel $A(v_1)B(v_2)$. Given an operator $\mathbf{K}$ with kernel $\bK(v_1,v_2)$,
we write
\begin{equation}\label{Kkminus}
K^{(-k)}(v)=\int_0^{\infty}\frac{\lambda^{k-1}}{(k-1)!}\bK(v,\lambda)\,d\lambda.
\end{equation}
Note that, if $\bK=\bA$ is given by (\ref{Akernel}) this is consistent with (\ref{Akminus}).

The basic functions that we will use are
\begin{equation}\label{Ai}
A_{i,\pm}(v)=\Ai(\xi_i+\eta_i^2+v)e^{\pm(\xi_i+v)\eta_i\pm2\eta_i^3/3},
\end{equation}
$i=1,2$. Define the operators $\bK_i$, $i=1,2$, on $L^2(\mathbb{R}_+)$ by
\begin{equation}\label{Ki}
\bK_i=\bA_{i,+}\bA_{i,-}.
\end{equation}
Note that, unless $\eta_i=0$, the operator $\bK_i$ is not symmetric, in fact
\begin{equation}\label{Kistar}
\bK_i^*=\bA_{i,-}\bA_{i,+},
\end{equation}
with kernel $\bK_i^*(v_1,v_2)=\bK_i(v_2,v_1)$. Recall the definition of the Airy kernel (\ref{Airykernel}). When $\eta_i=0$, then $\bK_i$ is the operator with kernel
\begin{equation}\label{KAiry}
\bK_{\Ai,\xi_i}(v_1,v_2)=\bK_{\Ai}(\xi_i+v_1,\xi_i+v_2),
\end{equation}
i.e. we have the Airy kernel shifted by $\xi_i$. If $\eta_i\neq 0$, then $\bK_i$ is a conjugation of the Airy kernel shifted by $\xi_i+\eta_i^2$,
\begin{equation}\label{KAiry2}
\bK_i(v_1,v_2)=e^{\eta_i(v_1-v_2)}\bK_{\Ai,\xi_i+\eta_i^2}(v_1,v_2).
\end{equation}

We will now define the quantities that will appear in our theorem. Let
\begin{equation}\label{r1}
r_1=r_1(\xi_1,\eta_1)=\Tr (\bI -\bK_1)^{-1}\bK_1,
\end{equation}
\begin{equation}\label{a0}
a_0=a_0(\xi_2,\eta_2)=\Tr (\bI -\bK_2^*)^{-1}\bA_{2,-}\otimes \bA_{2,+},
\end{equation}
\begin{equation}\label{a1}
a_1=a_1(\xi_2,\eta_2)=\Tr (\bI -\bK_2^*)^{-1}\bA_{2,-}^{(1)}\otimes \bA_{2,+}.
\end{equation}
Let $F_2$ denote the GUE Tracy-Widom distribution,
\begin{equation}\label{TW}
F_2(\xi)=\det(\bI-\bK_{\Ai,\xi})_{L^2(\mathbb{R}_+)}.
\end{equation}
From (\ref{KAiry2}), we see that
\begin{equation}\label{TW2}
\det(\bI-\bK_i)_{L^2(\mathbb{R}_+)}=F_2(\xi_i+\eta_i^2).
\end{equation}
Furthermore, for $\epsilon_1,\epsilon_2\in\{0,1\}$, $r,s\ge 0$, and $(\epsilon_1,\epsilon_2)\neq(0,0)$, we define
\begin{equation}\label{brs}
b_{r,s}(\epsilon_1,\epsilon_2)=\int_{\bbR_+^2}(\bA_{1,-}^{\epsilon_1}(\bI-\bK_1)^{-r}\bA_{1,+}^{\epsilon_2})(\lambda_1,\lambda_2)\frac{\lambda_2^s}{s!}\,d^2\lambda.
\end{equation}
and if $(\epsilon_1,\epsilon_2)=(0,0)$, we set
\begin{equation}\label{brs0}
b_{r,s}(0,0)=\int_{\bbR_+^2}(\bI-\bK_1)^{-r}\bK_1(\lambda_1,\lambda_2)\frac{\lambda_2^s}{s!}\,d^2\lambda.
\end{equation}
We will write
\begin{equation}\label{br}
b_{r}(\epsilon_1,\epsilon_2)=b_{r,0}(\epsilon_1,\epsilon_2),
\end{equation}
for $r\ge 0$, $\epsilon_1,\epsilon_2\in\{0,1\}$
Note that $b_{r,s}$ is a function of $\xi_1,\eta_1$. If we want to indicate this we write $b_{r,s}(\epsilon_1,\epsilon_2)(\xi_1,\eta_1)$.
Define,
\begin{align}\label{g1}
\psi_1(\xi_1,\eta_1)&=\xi_1-b_1(1,1)-b_1(0,0)+b_1(1,0)+b_1(0,1),\\
\psi_2(\xi_1,\eta_1)&=b_1(1,0)+b_1(0,1)-b_1(0,0)+b_2(1,1)+b_2(0,0)-b_2(1,0)-b_2(0,1),\notag
\end{align}
\begin{equation}\label{phi1}
\phi_1(\xi_1,\eta_1)=-b_{1,1}(1,1)-b_{1,1}(1,0)+b_{1,1}(0,0)-b_{0,1}(1,0)+b_{0,1}(0,1)+\xi_1[b_1(0,1)-b_1(1,1)]+\frac{\xi_1^2}{2}-\eta_1,
\end{equation}
and
\begin{align}\label{phi2}
\phi_2(\xi_1,\eta_1)&=b_{2,1}(1,1)+b_{2,1}(1,0)-b_{2,1}(0,1)-b_{2,1}(0,0)-b_{1,1}(1,0)+b_{1,1}(0,1)+b_{1,1}(0,0)\\&-b_{0,1}(1,0)-b_{0,1}(0,1)
+\xi_1[b_2(1,1)-b_2(0,1)+b_1(0,1)].\notag
\end{align}

We can now state our result for the long time limit.
\begin{theorem}\label{Thmalphazero}
We have the asymptotic formula,
\begin{equation}\label{Fttzero}
F_{tt}(\xi_1,\eta_1;\xi_2,\eta_2;\alpha)=F_2(\xi_1+\eta_1^2)F_2(\xi_2+\eta_2^2)(1+e_1\alpha+e_2\alpha^2)+O(\alpha^3)
\end{equation}
as $\alpha\to 0$. Here,
\begin{equation}\label{e1}
e_1=e_1(\xi_1,\eta_1,\xi_2,\eta_2)=a_0(\xi_2,\eta_2)ß[r_1(\xi_1,\eta_1)\psi_1(\xi_1,\eta_1)+\psi_2(\xi_1,\eta_1)],
\end{equation}
\begin{align}\label{e2}
e_2=e_2(\xi_1,\eta_1,\xi_2,\eta_2)&=r_1(\xi_1,\eta_1)\big[a_1(\xi_2,\eta_2)\phi_1(\xi_1,-\eta_1)+a_1(\xi_2,-\eta_2)\phi_1(\xi_1,\eta_1)\big]\\
&+a_1(\xi_2,\eta_2)\phi_2(\xi_1,-\eta_1)+a_1(\xi_2,-\eta_2)\phi_2(\xi_1,\eta_1).\notag
\end{align}
\end{theorem}

The theorem will be proved in section \ref{secThmalphazero}. It is possible to compute higher order terms using the same approach but the computations become
rather cumbersome so we stopped at the second order.

\begin{remark}
In the case $\eta_1=\eta_2=0$, the formulas for $e_1$ and $e_2$ can be simplified. When $\eta_1=\eta_2=0$, we write
\begin{equation*}
b_{r,s}(0)=b_{r,s}(0,0)(\xi_1,0)=b_{r,s}(1,1)(\xi_1,0),\quad b_{r,s}(1)=b_{r,s}(1,0)(\xi_1,0)=b_{r,s}(0,1)(\xi_1,0),
\end{equation*}
and $r_1(\xi_1)=r_1(\xi_1,0)$.
Then,
\begin{align}\label{g1spec}
\psi_1(\xi_1)&=\psi_1(\xi_1,0)=\xi_1-2b_1(0)+2b_1(1),
\\
\psi_2(\xi_1)&=\psi_2(\xi_1,0)=2b_1(1)-b_1(0)+2b_2(0)-2b_2(1),\notag
\end{align}
and
\begin{equation}\label{e1spec}
e_1(\xi_1,\xi_2)=\frac{F_2'(\xi_2)}{F_2(\xi_2)}[r_1(\xi_1)\psi_1(\xi_1)+\psi_2(\xi_1)].
\end{equation}
Also, with
\begin{equation}\label{a1spec}
a_1(\xi_2)=a_1(\xi_2,0),
\end{equation}
and
\begin{align}\label{phi1spec}
\phi_1(\xi_1)&=\phi_1(\xi_1,0)=-b_{1,1}(1)+\xi_1[b_1(1)-b_1(0)]+\frac 12\xi_1^2,\\
\phi_2(\xi_1)&=\phi_2(\xi_1,0)=-b_{1,1}(0)+\xi_1[b_1(1)+b_2(0)-b_2(1)],\notag
\end{align}
we get
\begin{equation}\label{e2spec}
e_2(\xi_1,\xi_2)=2a_1(\xi_2)[r_1(\xi_1)\phi_1(\xi_1)+\phi_2(\xi_1)].
\end{equation}
\end{remark}

It is possible to give express $a_0(\xi_2,\eta_2)$ and $a_1(\xi_2,0)$ in terms of the Tracy-Widom distribution. The following Lemma will be proved in section \ref{secLemmas}.
\begin{lemma}\label{Lema0}
We have the formulas,
\begin{equation}\label{a0TW}
a_0(\xi_2,\eta_2)=\frac{F_2'(\xi_2+\eta_2^2)}{F_2(\xi_2+\eta_2^2)},
\end{equation}
and
\begin{equation}\label{a1TW}
a_1(\xi_2,0)=-\frac{F_2''(\xi_2)}{2F_2(\xi_2)}.
\end{equation}
\end{lemma}

\subsection{Results for the short time limit}\label{subsectShort}

Next, we consider the short time limit $\alpha\to\infty$. In the limit $\alpha\to\infty$ we consider $h_1$ together with the random variable
\begin{equation}\label{hrv}
h=\alpha 'h_2-\alpha h_1,
\end{equation}
where, as above, $h_i=h_i(t_i,\xi_i)$, and
\begin{equation}\label{alphaprime}
\alpha'=(1+\alpha^3)^{1/3}=\left(\frac{t_2}{t_2-t_1}\right)^{1/3}.
\end{equation}
Note that $\alpha'=\alpha+\frac 1{3\alpha^2}+\dots$, so we could write $h=\alpha(h_2-h_1)$ in (\ref{hrv}). As $\alpha\to\infty$ the difference $h_2-h_1$ goes to zero
and multiplying by $\alpha$ turns out to be the right scaling to see a non-trivial limit. The fact that we use $\alpha'$ in (\ref{hrv}) is more technical. It is related to the
fact that $\alpha'$ appears naturally in the convolution equation
\begin{equation*}
\int_{-\infty}^\infty\Ai(a+\alpha x)\Ai(b-x)\,dx=\frac 1{\alpha'}\Ai\left(\frac{a+\alpha b}{\alpha'}\right).
\end{equation*}
We want to investigate the distribution function $\mathbb{P}[h_1\le\xi_1,h\le\xi]$ as $\alpha\to\infty$. From (\ref{ttheight}), we see that
\begin{align}
\mathbb{P}[h_1\le\xi_1,h\le\xi]&=\mathbb{P}[h_1\le\xi_1,\alpha'h_2-\alpha h_1\le\xi]\\
&=\int_{\{\xi_1'\le \xi_1,\xi_2'\le (\alpha\xi_1'+\xi)/\alpha'\}}\frac{\partial^2 F_{tt}}{\partial\xi_1\partial\xi_2}(\xi_1',\eta_1;\xi_2',\eta_2;\alpha)\,d\xi_1'd\xi_2'.
\notag
\end{align}
We can do the $\xi_2'$-integration and this gives the formula
\begin{equation}\label{honeh}
\mathbb{P}[h_1\le\xi_1,h\le\xi]=\int_{-\infty}^{\xi_1}\frac{\partial F_{tt}}{\partial\xi_1}(\xi_1',\eta_1;\frac{\xi+\alpha\xi_1'}{\alpha'},\eta_2;\alpha)\,d\xi_1'.
\end{equation}
To get a limit we also need to rescale $\eta_2$ as $\alpha\to\infty$ appropriately. Thus, we set 
\begin{equation}\label{xi2eta2}
\xi_2=\frac{\xi+\alpha\xi_1}{\alpha'},\quad\eta_2=\frac{\eta+\alpha^2\eta_1}{\alpha'^2},
\end{equation}
with $\xi_1$, $\xi$, $\eta_1$ and $\eta$ fixed as $\alpha\to\infty$. The first equation in (\ref{xi2eta2}) comes from (\ref{honeh}) and the second gives the rescaling we have to do
in $\eta_2$ to get a good limit.
To avoid analyzing the integration in (\ref{honeh}), we choose to study
\begin{equation}\label{dFxi1}
\frac{\partial}{\partial\xi_1}\mathbb{P}[h_1\le\xi_1,h\le\xi]=\frac{\partial F_{tt}}{\partial\xi_1}(\xi_1,\eta_1;\frac{\xi+\alpha\xi_1}{\alpha'},\frac{\eta+\alpha^2\eta_1}{\alpha'^2};\alpha).\end{equation}
We make some definitions in analogy with (\ref{Ai}) and (\ref{Ki}),
\begin{equation}\label{A0}
A_{0,\pm}(v)=\Ai(\xi+\eta^2+v)e^{\pm(\xi+v)\eta\pm\frac 23\eta^3},
\end{equation}
and
\begin{equation}\label{K0}
\bK_0=\bA_{0,+}\bA_{0,-}.
\end{equation}
Furthermore,
\begin{equation}\label{brtilde}
\tilde{b}_r(\epsilon_1,\epsilon_2)=\int_{\bbR_+^2}(\bA_{0,-}^{\epsilon_1}(\bI-\bK_0)^{-r}\bA_{0,+}^{\epsilon_2})(\lambda_1,\lambda_2)\,d^2\lambda
\end{equation}
for $(\epsilon_1,\epsilon_2)\neq(0,0)$, $r\ge 0$, and
\begin{equation}\label{brtilde0}
\tilde{b}_r(0,0)=\int_{\bbR_+^2}(\bI-\bK_0)^{-r}\bK_0(\lambda_1,\lambda_2)\,d^2\lambda,
\end{equation}
which are functions of $\xi,\eta$.
Let,
\begin{equation}\label{g2}
\psi(\xi,\eta)=\xi+\tilde{b}_1(0,1)+\tilde{b}_1(1,0)-\tilde{b}_1(0,0)-\tilde{b}_1(1,1).
\end{equation}
We can now formulate our theorem which will be proved in section \ref{secThmalphainfty}.

\begin{theorem}\label{Thmalphainfty}
We have the following limit,
\begin{equation}\label{Fttinfty}
\lim_{\alpha\to\infty}\frac{\partial F_{tt}}{\partial\xi_1}(\xi_1,\eta_1;\frac{\xi+\alpha\xi_1}{\alpha'},\frac{\eta+\alpha^2\eta_1}{\alpha'^2};\alpha)
=F_2'(\xi_1+\eta_1^2)\frac{\partial}{\partial\xi}\big(F_2(\xi+\eta^2)\psi(\xi,\eta)\big).
\end{equation}
\end{theorem}

\begin{remark}
It is possible with more effort to not just compute the limit but also compute further terms in an $1/\alpha$ expansion, see (\ref{GHfor}) below.
\end{remark}

\begin{remark} 
We do of course expect that
\begin{equation*}
\lim_{\alpha\to\infty}F_{tt}(\xi_1,\eta_1;\frac{\xi+\alpha\xi_1}{\alpha'},\frac{\eta+\alpha^2\eta_1}{\alpha'^2};\alpha)
=F_2(\xi_1+\eta_1^2)\frac{\partial}{\partial\xi}\big(F_2(\xi+\eta^2)\psi(\xi,\eta)\big),
\end{equation*}
but to prove this rigorously would require further estimates.
\end{remark}

\begin{remark}
If $\eta=0$, then $\bA_{0,+}=\bA_{0,-}:=\bA_0$ and $\bK_0=\bA_0^2=\bK_{\Ai,\xi}$, and hence
\begin{equation*}
\tilde{b}_r(0):=\tilde{b}_r(0,0)=\tilde{b}_r(1,1), \quad \tilde{b}_r(1):=\tilde{b}_r(1,0)=\tilde{b}_r(0,1),
\end{equation*}
where
\begin{align}\label{btildespecial}
\tilde{b}_1(0)&=\int_{\bbR_+^2}(\bI-\bK_{\Ai,\xi})^{-1}\bK_{\Ai,\xi}(\lambda_1,\lambda_2)\,d\lambda,\\
\tilde{b}_1(1)&=\int_{\bbR_+^2}(\bI-\bK_{\Ai,\xi})^{-1}\bA_0(\lambda_1,\lambda_2)\,d\lambda.\notag
\end{align}
Thus, we see that
\begin{equation}\label{psi2special}
\psi(\xi):=\psi(\xi,0)=\xi+2\tilde{b}_1(1)-2\tilde{b}_1(0),
\end{equation}
and
\begin{equation*}
\frac{d}{d\xi}\left(F_2(\xi)\psi(\xi)\right)=F_0(\xi),
\end{equation*}
is the Baik-Rains distribution.
\end{remark}

\subsection{A formula for the two-time distribution}\label{sectShort}

We now recall some formulas from Section 6 in \cite{Jott} which we will use to prove our results. First, recall the notation
\begin{equation}\label{Gxieta}
G_{\xi,\eta}(z)=e^{\frac 13 z^3+\eta z^2-\xi z}.
\end{equation}
and
\begin{equation}\label{deltaxideltaeta}
\Delta\xi=\alpha'\xi_2-\alpha\xi_1,\quad\Delta\eta=\alpha'^2\eta_2-\alpha^2\eta_1.
\end{equation}
Let $\Gamma_{D}$ denote the vertical contour $t\mapsto D+it$, $t\in\bbR$.
Note that if $d,D>0$, then
\begin{equation}\label{GAk1}
\frac 1{\pii}\int_{\Gamma_D}z^kG_{\xi_i+v,\eta_i}(z)\,dz=A_{i,+}^{(k)}(v),
\end{equation}
and
\begin{equation}\label{GAk2}
\frac 1{\pii}\int_{\Gamma_{-d}}\frac{\zeta^k}{G_{\xi_i+v,\eta_i}(\zeta)}\,d\zeta=(-1)^kA_{i,-}^{(k)}(v),
\end{equation}
for all $k\in\mathbb{Z}$. This explains the appearance of $A_{i,\pm}^{(k)}$. The formula (\ref{GAk1}) is straightforward to show from the integral formula
\begin{equation}\label{GAk3}
\frac 1{\pii}\int_{\Gamma_{D}}e^{z^3/3-\xi z}\,dz=\Ai(\xi),
\end{equation}
for $D>0$, by differentiation. Then, (\ref{GAk2}) follows by setting $\zeta=-z$. 

On the space
\begin{equation}\label{Yspace}
Y=L^2(\mathbb{R}_+)\oplus L^2(\mathbb{R}_+),
\end{equation}
we consider the matrix operator kernel $\bQ=\bQ(u)$ given by
\begin{equation}\label{Qformula1}
\bQ(u)=\begin{pmatrix} (2-u-u^{-1})\bk_1+(u-1)(\bk_2+\bk_5)+(u-1)\bM_3-u\bM_2 & (u+u^{-1}-2)\bk_3+(1-u)\bk_4  \\
                                    (1-u^{-1})\bk_6-\bk_7          & (u^{-1}-1)\bM_1
                                              
                    \end{pmatrix},
 \end{equation}
where $\bk_i$ and $\bM_i$ are given below.
By Proposition 6.1 in \cite{Jott}, we then have the formula
\begin{equation}\label{Fttformula1}
F_{tt}(\xi_1,\eta_1;\xi_2,\eta_2;\alpha)=\frac 1{\pii}\int_{\gamma_r}\det(\bI+\bQ(u))_Y\frac{du}{u-1},
\end{equation}
where $\gamma_r$ is a circle around the origin with radius $r>1$.

\begin{remark} Below all Fredholm determinants of matrix operators will be on the space $Y$ and all scalar Fredholm determinants will be on the space
$L^2(\bbR_+)$. We will not always indicate the dependence on the space.
\end{remark}

Let $d,D>0$. Then the kernels appearing in (\ref{Qformula1}) are given by the following formulas. Let
\begin{equation}\label{M1}
\bM_1(v_1,v_2)=\frac{e^{\delta(v_1-v_2)}}{(\pii)^2}\int_{\Gamma_D}dz\int_{\Gamma_{-d}}d\zeta\frac{G_{\xi_1+v_1,\eta_1}(z)}{G_{\xi_1+v_2,\eta_1}(\zeta)(z-\zeta)},
\end{equation}
\begin{equation}\label{M2}
\bM_2(v_1,v_2)=\frac{1}{(\pii)^2\alpha'}\int_{\Gamma_D}dz\int_{\Gamma_{-d}}d\zeta\frac{G_{\xi_2+v_2/\alpha',\eta_2}(z)}{G_{\xi_2+v_1/\alpha',\eta_2}(\zeta)(z-\zeta)},
\end{equation}
and
\begin{equation}\label{M3}
\bM_3(v_1,v_2)=\frac{1}{(\pii)^2}\int_{\Gamma_D}dz\int_{\Gamma_{-d}}d\zeta\frac{G_{\Delta\xi+v_2,\Delta\eta}(z)}{G_{\Delta\xi+v_1,\Delta\eta}(\zeta)(z-\zeta)}.
\end{equation}
Assume the following condition on the horizontal positions of the contours,
\begin{equation}\label{dDconditions}
0<d_1<\alpha d_2<d_3,\quad 0<D_1<\alpha D_2<D_3.
\end{equation}
Define,
\begin{align}\label{k1}
&\bk_1(v_1,v_2)\\&=\frac{\alpha}{(\pii)^4}\int_{\Gamma_{D_3}}dz\int_{\Gamma_{D_2}}dw\int_{\Gamma_{-d_3}}d\zeta\int_{\Gamma_{-d_2}}d\omega
\frac{G_{\xi_1,\eta_1}(z)G_{\Delta\xi+v_2,\Delta\eta}(w)}{G_{\xi_1,\eta_1}(\zeta)G_{\Delta\xi+v_1,\Delta\eta}(\omega)(z-\zeta)(z-\alpha w)(\alpha\omega-\zeta)},\notag
\end{align}
\begin{equation}\label{k2}
\bk_2(v_1,v_2)=\frac{\alpha}{(\pii)^3}\int_{\Gamma_{D_3}}dz\int_{\Gamma_{D_2}}dw\int_{\Gamma_{-d_2}}d\omega
\frac{G_{\xi_1,\eta_1}(z)G_{\Delta\xi+v_2,\Delta\eta}(w)}{G_{\xi_2+v_1/\alpha',\eta_2}(\omega)(\alpha'z-\alpha\omega)(z-\alpha w)},
\end{equation}
\begin{equation}\label{k3}
\bk_3(v_1,v_2)=\frac{\alpha e^{-\delta v_2}}{(\pii)^2}\int_{\Gamma_{-d_3}}d\zeta\int_{\Gamma_{-d_2}}d\omega
\frac 1{G_{\xi_1+v_2,\eta_1}(\zeta)G_{\Delta\xi+v_1,\Delta\eta}(\omega)(\alpha\omega-\zeta)},
\end{equation}
\begin{equation}\label{k4}
\bk_4(v_1,v_2)=\frac{\alpha e^{-\delta v_2}}{\alpha'\pii}\int_{\Gamma_{-d_2}}\frac{d\omega}{G_{\xi_2+(v_1+\alpha v_2)/\alpha',\eta_2}(\omega)},
\end{equation}
\begin{equation}\label{k5}
\bk_5(v_1,v_2)=\frac{\alpha}{(\pii)^3}\int_{\Gamma_{D_2}}dw\int_{\Gamma_{-d_3}}d\zeta\int_{\Gamma_{-d_2}}d\omega
\frac{G_{\xi_2+v_2/\alpha',\eta_2}(w)}{G_{\xi_1,\eta_1}(\zeta)G_{\Delta\xi+v_1,\Delta\eta}(\omega)(\alpha w-\alpha'\zeta)(\alpha\omega-\zeta)},
\end{equation}
\begin{equation}\label{k6}
\bk_6(v_1,v_2)=\frac{e^{\delta v_1}}{(\pii)^4}\int_{\Gamma_{D_3}}dz_1\int_{\Gamma_{D_1}}dz_2\int_{\Gamma_{D_2}}dw\int_{\Gamma_{-d_1}}d\zeta
\frac{G_{\xi_1,\eta_1}(z_1)G_{\xi_1+v_1,\eta_1}(z_2)G_{\Delta\xi+v_2,\Delta\eta}(w)}{G_{\xi_1,\eta_1}(\zeta)(z_1-\zeta)(z_2-\zeta)(z_1-\alpha w)},
\end{equation}
and
\begin{equation}\label{k7}
\bk_7(v_1,v_2)=\frac{e^{\delta v_1}}{(\pii)^3}\int_{\Gamma_{D_1}}dz\int_{\Gamma_{D_2}}dw\int_{\Gamma_{-d_1}}d\zeta
\frac{G_{\xi_1+v_1,\eta_1}(z)G_{\xi_2+v_2/\alpha',\eta_2}(w)}{G_{\xi_1,\eta_1}(\zeta)(\alpha w-\alpha'\zeta)(z-\zeta)}.
\end{equation}

\section{Proof of the long time expansion}\label{secThmalphazero}
In this section we will prove Theorem \ref{Thmalphazero}. In this case $\xi_1, \eta_1,\xi_2$ and $\eta_2$ are fixed and $\Delta\xi,\Delta\eta$ are given by (\ref{deltaxideltaeta}).
The operators $\bM_i$ and $\bk_i$ given by (\ref{M1}) to (\ref{M3}) and (\ref{k1}) to (\ref{k7}) can be expanded in powers of $\alpha$,
\begin{equation}\label{Miexp}
\bM_i=\bM_{i,0}+\bM_{i,1}\alpha+\frac 12\bM_{i,2}\alpha^2+\tilde{\bM}_{i,3}\alpha^3
\end{equation}
and
\begin{equation}\label{kiexp}
\bk_i=\bk_{i,0}+\bk_{i,1}\alpha+\frac 12\bk_{i,2}\alpha^2+\tilde{\bk}_{i,3}\alpha^3,
\end{equation}
where $\bM_{i,s}$ and $\bk_{i,s}$ are trace class operators that do not depend on $\alpha$, and $\tilde{\bM}_{i,3}$ and $\tilde{\bk}_{i,3}$ are bounded as $\alpha\to 0$.
We will not discuss this in any detail. It follows by analyzing the $\alpha$-expansion of the formulas in the right sides of (\ref{M1}) to (\ref{k7}). We show in Lemma \ref{LemkiMiformulas} that $\bM_{i,s}$, $1\le s\le 2$ an $\bk_{i,s}$, $0\le s\le 2$, are finite rank operators. The fact that the remainder parts of (\ref{Miexp}) and (\ref{kiexp}),
$\tilde{\bM}_{i,3}\alpha^3$, $\tilde{\bk}_{i,3}\alpha^3$, have the desired properties also follows by analyzing the right sides of (\ref{M1}) to (\ref{k7}). Note that they can
all be expressed in terms of Airy functions or integrals of Airy functions, see \cite{Jott}. See also \cite{GeLaZu} for a discussion of perturbation theory of Fredholm determinants.
We see that
\begin{equation}\label{DiffMk}
\bM_{i,s}=\left.\frac{\partial^s}{\partial\alpha^s}\right|_{\alpha=0}\bM_i,\quad \bk_{i,s}=\left.\frac{\partial^s}{\partial\alpha^s}\right|_{\alpha=0}\bk_i.
\end{equation}
It follows from (\ref{Qformula1}), (\ref{Miexp}) and (\ref{kiexp}) that
\begin{equation}\label{Qexp}
\bQ=\bQ_{0}+\bQ_{1}\alpha+\frac 12\bQ_{2}\alpha^2+\tilde{\bQ}_{3}\alpha^3,
\end{equation}
where $\bQ_{s}$ are finite rank operators that do not depend on $\alpha$ and $\tilde{\bQ}_{3}$ is bounded as $\alpha\to 0$. From (\ref{Qformula1}), we see that
\begin{equation}\label{Qr}
\bQ_r=\begin{pmatrix} (2-u-u^{-1})\bk_{1,r}+(u-1)(\bk_{2,r}+\bk_{5,r})+(u-1)\bM_{3,r}-u\bM_{2,r} & (u+u^{-1}-2)\bk_{3,r}+(1-u)\bk_{4,r}  \\
                                    (1-u^{-1})\bk_{6,r}-\bk_{7,r}          & (u^{-1}-1)\bM_{1,r}
                                              
                    \end{pmatrix},
 \end{equation}
 for $r=0,1,2$. 
 
 In order to give a formula for $\bQ_r$, we need to compute the derivatives in (\ref{DiffMk}). Write
 \begin{equation}\label{cdelta}
c_{\delta}(v)=e^{-\delta v},
\end{equation}
where $\delta>0$. Note that, by (\ref{GAk1}) and (\ref{GAk2}),
\begin{equation}\label{M1delta}
\bM_1(v_1,v_2)=c_{-\delta}(v_1)\bK_1(v_1,v_2)c_\delta(v_2).
\end{equation}
Differentiation gives
\begin{equation*}
\frac{\partial}{\partial\alpha}\det(\bI+\bQ)=\det(\bI+\bQ)\Tr(\bI+\bQ)^{-1}\frac{\partial\bQ}{\partial\alpha},
\end{equation*}
and
\begin{align*}
\frac{\partial^2}{\partial\alpha^2}\det(\bI+\bQ)&=\det(\bI+\bQ)\left(\Tr(\bI+\bQ)^{-1}\frac{\partial\bQ}{\partial\alpha}\right)^2-
\det(\bI+\bQ)\Tr(\bI+\bQ)^{-1}\frac{\partial\bQ}{\partial\alpha}(\bI+\bQ)^{-1}\frac{\partial\bQ}{\partial\alpha}\\
&+\det(\bI+\bQ)\Tr(\bI+\bQ)^{-1}\frac{\partial^2\bQ}{\partial\alpha^2}.
\end{align*}
Using (\ref{Qexp}), we see that
\begin{equation*}
\left.\frac{\partial}{\partial\alpha}\right|_{\alpha=0}\det(\bI+\bQ)=\det(\bI+\bQ_0)\Tr(\bI+\bQ_0)^{-1}\bQ_1,
\end{equation*}
and
\begin{align*}
\left.\frac{\partial^2}{\partial\alpha^2}\right|_{\alpha=0}\det(\bI+\bQ)&=\det(\bI+\bQ_0)\left[\left(\Tr(\bI+\bQ_0)^{-1}\bQ_1\right)^2\right.
-\Tr(\bI+\bQ_0)^{-1}\bQ_1(\bI+\bQ_0)^{-1}\bQ_1\\
&+\left.\Tr(\bI+\bQ_0)^{-1}\bQ_2\right].
\end{align*}
Consequently,
\begin{align}\label{detQexp}
\det(\bI+\bQ)&=\det(\bI+\bQ_0)\left[(\Tr(\bI+\bQ_0)^{-1}\bQ_1)\alpha+\frac 12\left[(\Tr(\bI+\bQ_0)^{-1}\bQ_1)^2\right.\right.\\
&\left.\left.-\Tr(\bI+\bQ_0)^{-1}\bQ_1(\bI+\bQ_0)^{-1}\bQ_1+\Tr(\bI+\bQ_0)^{-1}\bQ_2\right]\alpha^2+O(\alpha^3)\right].\notag
\end{align}
For $i=1,2$ we write
\begin{equation}\label{Li}
\bL_i=(\bI-\bK_i)^{-1},\quad\bL_i^*=(\bI-\bK_i^*)^{-1},\tilde{\bL}_i=(\bI-\bM_i)^{-1},
\end{equation}
and
\begin{equation}\label{Rop}
\bR=(\bI-\bM_1)^{-1}\bM_1.
\end{equation}
We also write,
\begin{equation}\label{L0}
\tilde{\bL}_i=(\bI-\bM_i)^{-1}.
\end{equation}
Note that,
\begin{equation}\label{L1}
\tilde{\bL}_1(v_1,v_2)=c_{-\delta}(v_1)\bK_1(v_1,v_2)c_{\delta}(v_2).
\end{equation}

\begin{lemma}\label{LemQ0exp}
We have that
\begin{equation}\label{detQ0for}
\det(\bI+\bQ_0)=F_2(\xi_1+\eta_1^2)F_2(\xi_2+\eta_2^2)\det(I+u^{-1}\bR).
\end{equation}
If $|u|$ is sufficiently large, then
\begin{equation}\label{detRexp}
\det(I+u^{-1}\bR)=1+\sum_{k=1}^\infty r_ku^{-k},
\end{equation}
where $r_1$ is given by (\ref{r1}). Furthermore, for $|u|$ sufficiently large,
\begin{equation}\label{Q0invexp}
(\bI+\bQ_0)^{-1}=\sum_{k=0}^\infty \bQ_0(k)u^{-k},
\end{equation}
where
\begin{equation}\label{Q00}
\bQ_0(0)=\begin{pmatrix}  \bL_2^*&  0 \\
                                    \tilde{\bL}_1(\bk_{7,0}-\bk_{6,0})\bL_2^*        & \tilde{\bL}_1
                                              
                    \end{pmatrix},
\end{equation}
and
\begin{equation}\label{Q0k}
\bQ_0(k)=\begin{pmatrix} 0 &  0 \\
                                    (-1)^{k-1}\left(\bR^{k-1}\tilde{\bL}_1\bk_{6,0}\bL_2^* -\bR^k\tilde{\bL}_1(\bk_{7,0}-\bk_{6,0})\bL_2^*\right)      & (-1)^k\bR^k\tilde{\bL}_1
                                              
                    \end{pmatrix},
\end{equation}
for $k\ge 1$.
\end{lemma}

\begin{Proof}
See section \ref{secLemmas}.
\end{Proof}

From (\ref{Qr}), we see that we can write
\begin{equation}\label{Qru}
\bQ_r=\bQ_r(-1)u+\bQ_r(0)+\bQ_r(1)u^{-1},
\end{equation}
where
\begin{align}\label{Qrformulas}
\bQ_r(-1)&=
\begin{pmatrix}  -\bk_{1,r}+\bk_{2,r}+\bk_{5,r}+\bM_{3,r} & \bk_{3,r}-\bk_{4,r}\\
                                   0     & 0
                                              \end{pmatrix},
                                              \\
                                             \bQ_r(0)&=\begin{pmatrix}  2\bk_{1,r}-\bk_{2,r}-\bk_{5,r}-\bM_{3,r}&  -2\bk_{3,r}+\bk_{4,r}\\
                               \bk_{6,r}-\bk_{7,r}       & 0
                                              \end{pmatrix},\notag\\
                                              \bQ_r(1)&=
                                              \begin{pmatrix}  -\bk_{1,r}& \bk_{3,r}\\
                              -\bk_{6,r}        & 0
                                              \end{pmatrix}.\notag
\end{align}
Define $\bP_r$ and $\bP_r(k)$, $k\ge -1$, by
\begin{equation}\label{Prkdef}
\bP_r=(\bI+\bQ_0)^{-1}\bQ_r=\sum_{k=-1}^\infty \bP_r(k)u^{-k},
\end{equation}
for $r=1,2$. It follows from  (\ref{Q0invexp}) and (\ref{Prkdef}) that
\begin{align}\label{Prform}
\bP_r(-1)&=\bQ_0(0)\bQ_r(-1),\\
\bP_r(0)&=\bQ_0(0)\bQ_r(0)+\bQ_0(1)\bQ_r(-1),\notag\\
\bP_r(k)&=\bQ_0(k-1)\bQ_r(1)+\bQ_0(k)\bQ_r(0)+\bQ_0(k+1)\bQ_r(-1),\notag
\end{align}
for $r=1,2$, $k\ge 1$. Define,
\begin{equation}\label{sigmakl}
\sigma(k,\ell)=\Tr\bP_1(k)\bP_1(\ell)-\Tr\bP_1(k)\Tr\bP_1(\ell),
\end{equation}
$k,\ell\ge -1$.

We can now give expressions for $e_1$ and $e_2$ in Theorem \ref{Thmalphazero}.
\begin{lemma}\label{Leme1e2formula}
We have the following formulas for $e_1$ and $e_2$ in (\ref{Fttzero}),
\begin{equation}\label{e1formula}
e_1=r_1\Tr\bP_1(-1)+\Tr\bP_1(-1)+\Tr\bP_1(0),
\end{equation}
and
\begin{align}\label{e2formula}
e_2&=-r_2\sigma(-1,-1)+r_1\left(\frac 12\Tr\bP_2(-1)-\frac 12\sigma(-1,-1)-\sigma(-1,0)\right)-\frac 12\sigma(-1,-1)\\
&-\sigma(-1,0)-\frac 12\sigma(0,0)-\sigma(-1,1)+\frac 12\Tr\bP_2(-1)+\frac 12\Tr\bP_2(0),\notag
\end{align}
with $r_1$ and $r_2$ given by (\ref{detRexp}).
\end{lemma}

We will prove the Lemma in section \ref{secLemmas}.

In order to prove Theorem \ref{Thmalphazero} we have to compute the expressions in (\ref{e1formula}) and (\ref{e2formula}). 
To write our formulas, it will be convenient to introduce some more notation. In analogy with (\ref{brs}), we define
\begin{equation}\label{brsstar}
b_{r,s}^*(\epsilon_1,\epsilon_2)=\int_{\bbR_+^2}\frac{\lambda_1^s}{s!}(\bA_{1,-}^{\epsilon_1}(\bI-\bK_1)^{-r}\bA_{1,+}^{\epsilon_2})(\lambda_1,\lambda_2)\,d^2\lambda.
\end{equation}
for $(\epsilon_1,\epsilon_2)\neq (0,0)$.
If $(\epsilon_1,\epsilon_2)=(0,0)$ and $r,s\ge 0$, we let
\begin{equation}\label{brs0star}
b_{r,s}^*(0,0)=\int_{\bbR_+^2}\frac{\lambda_1^s}{s!}(\bI-\bK_1)^{-r}\bK_1(\lambda_1,\lambda_2)\,d^2\lambda.
\end{equation}
Also, in analogy with (\ref{a1}), we define
\begin{equation}\label{a1star}
a_1^*=a_1^*(\xi_2,\eta_2)=\Tr (\bI -\bK_2^*)^{-1}\bA_{2,-}\otimes \bA_{2,+}^{(1)}.
\end{equation}
The quantities $b_{r,s}^*$ and $a_1^*$ can be related to $b_{r,s}$ and $a_1$. In fact, we have the following Lemma which we will prove in section \ref{secLemmas}.
\begin{lemma}\label{Lembrsrelation}
We have the formulas
\begin{equation}\label{brsbrsstar}
b_{r,s}^*(\epsilon_1,\epsilon_2)(\xi_1,\eta_1)=b_{r,s}(\epsilon_2,\epsilon_1)(\xi_1,-\eta_1),
\end{equation}
\begin{equation}\label{b0sb0sstar}
b_{0,s}^*(0,1)=b_{0,s}(0,1),\quad b_{0,s}^*(1,0)=b_{0,s}(1,0),
\end{equation}
\begin{equation}\label{br0br0star}
b_{r,0}^*(\epsilon_1,\epsilon_2)=b_{r,0}(\epsilon_1,\epsilon_2),
\end{equation}
\begin{equation}\label{a1a1star}
a_1^*(\xi_1,\eta)=a_1(\xi,-\eta),
\end{equation}
and
\begin{equation}\label{r1rel}
r_1(\xi_1,-\eta)=r_1(\xi_1,\eta_1).
\end{equation}
\end{lemma}

The next Lemma gives expressions for the quantities we will need.
In some formulas, we will use the convention that
\begin{equation}\label{A0minusconv}
(A_{i,\pm}^0)^{(-1-s)}(\lambda)=\frac{\lambda^s}{s!}.
\end{equation}

\begin{lemma}\label{LemkiMiformulas}
We have the following formulas,
\begin{equation}\label{M1for}
\bM_{1,0}=\bM_1,\quad \bM_{1,1}=\bM_{1,2}=0,
\end{equation}
\begin{equation}\label{M2for}
\bM_{2,0}=\bK_2^*,\quad \bM_{2,1}=\bM_{2,2}=0,
\end{equation}
\begin{align}\label{M3for}
\bM_{3,0}&=\bK_2^*,\quad \bM_{3,1}= \xi_1A_{2,-}\otimes A_{2,+},\\  
\bM_{3,2}&=(\xi_1^2+2\eta_1)A_{2,-}^{(1)}\otimes A_{2,+}+(\xi_1^2-2\eta_1)A_{2,-}\otimes A_{2,+}^{(1)},\notag
\end{align}
\begin{align}\label{k1for}
\bk_{1,0}&=0,\quad\bk_{1,1}=b_0(0,0)A_{2,-}\otimes A_{2,+},\\  
\bk_{1,2}&=2(b_{0,1}^*(1,1)+\xi_1b_0(0,0))A_{2,-}^{(1)}\otimes A_{2,+}+2(b_{0,1}(1,1)+\xi_1b_0(0,0))A_{2,-}\otimes A_{2,+}^{(1)},\notag
\end{align}
\begin{align}\label{k2for}
\bk_{2,0}&=0,\quad\bk_{2,1}=b_0(0,1)A_{2,-}\otimes A_{2,+},\\  
\bk_{2,2}&=-4b_{0,1}(0,1)A_{2,-}^{(1)}\otimes A_{2,+}+(4b_{0,1}(1,1)+2\xi_1b_0(0,1))A_{2,-}\otimes A_{2,+}^{(1)},\notag
\end{align}
\begin{equation}\label{k3for}
\bk_{3,0}=0,\quad\bk_{3,1}=A_{2,-}\otimes A_{1,-}^{(-1)}c_\delta,\quad 
\bk_{3,2}=2A_{2,-}^{(1)}\otimes (A_{1,-}^{(-2)}+\xi_1A_{1,-}^{(-1)})c_\delta,
\end{equation}
\begin{equation}\label{k4for}
\bk_{4,0}=0,\quad\bk_{4,1}=A_{2,-}\otimes (A_{1,-}^0)^{(-1)}c_\delta,\quad
\bk_{4,2}=-2A_{2,-}\otimes(A_{1,-}^0)^{(-2)} c_\delta,
\end{equation}
\begin{equation}\label{k5for}
\bk_{5,0}=0,\quad\bk_{5,1}=b_0(1,0)A_{2,-}\otimes A_{2,+},\quad
\bk_{5,2}=(4b_0(1,0)+2\xi_1b_0(1,0))A_{2,-}^{(1)}\otimes A_{2,+},
\end{equation}
\begin{align}\label{k6for}
\bk_{6,0}&=c_{-\delta}(K_1A_{1,+})^{(-1)}\otimes A_{2,+},
\quad\bk_{6,1}=c_{-\delta}((K_1A_{1,+})^{(-2)}+\xi_1(K_1A_{1,+})^{(-1)})\otimes A_{2,+}^{(1)},\\  
\bk_{6,2}&=c_{-\delta}(2(K_1A_{1,+})^{(-3)}+2\xi_1(K_1A_{1,+})^{(-2)}+(\xi_1^2-2\eta_1)(K_1A_{1,+})^{(-1)})\otimes A_{2,+}^{(2)},\notag
\end{align}
\begin{equation}\label{k7for}
\bk_{7,0}=c_{-\delta}K_1^{(-1)}\otimes A_{2,+},
\quad
\bk_{7,1}=-c_{-\delta}K_1^{(-2)}\otimes A_{2,+}^{(1)},
\quad
\bk_{7,2}=-2c_{-\delta}K_1^{(-3)}\otimes A_{2,+}^{(2)}.
\end{equation}
\end{lemma}
The Lemma will be proved below in section \ref{secLemmas}.

It will be convenient to use the following expressions
which occur in the computations. For $\epsilon_1,\epsilon_2$, $r\ge 1$, $s\ge 0$, we define
\begin{equation}\label{grs}
g_{r,s}(\epsilon_1,\epsilon_2)=\Tr(\bI-\bK_1)^{-1}\bK_1^{r-1}(\bK_1\bA_{1,+}^{\epsilon_2})^{(-1-s)}\otimes(\bA_{1,-}^{\epsilon_1})^{(-1)},
\end{equation}
and
\begin{equation}\label{grsstar}
g_{r,s}^*(\epsilon_1,\epsilon_2)=\Tr(\bI-\bK_1)^{-1}\bK_1^{r-1}(\bK_1\bA_{1,+}^{\epsilon_2})^{(-1)}\otimes(\bA_{1,-}^{\epsilon_1})^{(-1-s)}.
\end{equation}

The quantities (\ref{grs}) can be expressed in terms of $b_{r,s}$. This is stated in the next Lemma, which is proved in section \ref{secLemmas}.
\begin{lemma}\label{Lemgrsformula}
For $r\ge1$, $s\ge 0$ and $(\epsilon_1,\epsilon_2)\neq(0,0)$, we have the formula,
\begin{equation}\label{grsfor}
g_{r,s}(\epsilon_1,\epsilon_2)=\sum_{k=0}^r(-1)^{r-k}\binom{r}{k}b_{k,s}(\epsilon_1,\epsilon_2).
\end{equation}
If $(\epsilon_1,\epsilon_2)=(0,0)$, we have instead,
\begin{equation}\label{grs0for}
g_{r,s}(0,0)=\sum_{k=1}^r(-1)^{r-k}\binom{r-1}{k-1}b_{k,s}(0,0).
\end{equation}
\end{lemma}
There are completely analogous statements for $g_{r,s}^*$, with $b_{r,s}^*$ instead of $b_{r,s}$ in the right side.

\begin{proof}[Proof of Theorem \ref{Thmalphazero}]
The proof is a lengthy but straightforward computation. We will explain how the computation can be done, but we will not give all the details.
In order not to get too long formulas, we need to introduce some shorthand notation for objects that occur in the computations. Define the functions,
\begin{align}\label{Ckdef}
C_1(v)&=c_{-\delta}(v)(\bK_1\bA_{1,+})^{-1}(v),\\
C_2(v)&=c_{-\delta}(v)(\bK_1^{(-1)}(v)-(\bK_1\bA_{1,+})^{-1}(v)),\notag\\
C_3(v)&=-c_{-\delta}(v)((\bK_1\bA_{1,+})^{(-2)}(v)+\xi_1(\bK_1\bA_{1,+})^{(-1)}(v)),\notag\\
C_4(v)&=A_{1,-}^{(-1)}(v)c_{\delta}(v),\notag\\
C_5(v)&=(-2A_{1,-}^{(-1)}(v)+(A_{1,-}^0)^{(-1)}(v))c_{\delta}(v),\notag\\
C_6(v)&=c_{-\delta}(v)((\bK_1\bA_{1,+})^{(-2)}(v)+\xi_1(\bK_1\bA_{1,+})^{(-1)}(v)+\bK_1^{(-1)}(v)),\notag\\
C_7(v)&=(A_{1,-}^{(-1)}(v)-(A_{1,-}^0)^{(-1)}(v))c_{\delta}(v),\notag\\
C_8(v)&=(-4A_{1,-}^{(-2)}(v)-4\xi_1(A_{1,-})^{(-1)}(v)-2(A_{1,-}^0)^{(-2)}(v))c_{\delta}(v),\notag\\
C_9(v)&=c_{-\delta}(v)(2(\bK_1\bA_{1,+})^{(-3)}(v)+2\xi_1(\bK_1\bA_{1,+})^{(-2)}(v)+(\xi_1^2-2\eta_1)(\bK_1\bA_{1,+})^{(-1)}(v)+2\bK_1^{(-3)}(v)),\notag\\
C_{10}(v)&=(2A_{1,-}^{(-2)}(v)+2\xi_1(A_{1,-})^{(-1)}(v)+2(A_{1,-}^0)^{(-2)}(v))c_{\delta}(v),\notag\\
C_{11}(v)&=C_1(v)-(\bR C_2)(v),\notag\\
C_{12}(v)&=-(\bR C_{11})(v)=-(\bR C_1)(v)+(\bR^2 C_2)(v).\notag
\end{align}
We also define the functions
\begin{align}\label{Bkdef}
B_1&=b_1A_{2,-},\quad B_2=b_2A_{2,-}, \quad B_3=b_3A_{2,-}^{(1)}, \quad B_4=b_4A_{2,-},\\
B_5&=b_5A_{2,-}^{(1)},\quad B_6=b_6A_{2,-}, \quad B_7=b_7A_{2,-},\notag
\end{align}
where
\begin{align}\label{bkdef}
b_1&=-b_0(1,1)+b_0(1,0)+b_0(0,1)+\xi_1,\\
b_2&=2b_0(1,1)-b_0(1,0)-b_0(0,1)-\xi_1,\notag\\
b_3&=4b_{0,1}^*(1,1)+4\xi_1b_0(1,1)-2b_{0,1}(1,0)-2\xi_1b_0(1,0)+2b_{0,1}(0,1)-\xi_1^2-2\eta_1,\notag\\
b_4&=4b_{0,1}(1,1)+4\xi_1b_0(1,1)+2b_{0,1}(1,0)-2b_{0,1}(0,1)-2\xi_1b_0(0,1)-\xi_1^2+2\eta_1,\notag\\
b_5&=-2b_{0,1}^*(1,1)-2\xi_1b_0(1,1)+4b_{0,1}(1,0)+2\xi_1b_0(1,0)-4b_{0,1}(0,1)+\xi_1^2+2\eta_1,\notag\\
b_6&=-2b_{0,1}(1,1)-2\xi_1b_0(1,1)-4b_{0,1}(1,0)+4b_{0,1}(0,1)+2\xi_1b_0(0,1)+\xi_1^2-2\eta_1,\notag\\
b_7&=-b_0(1,1).\notag
\end{align}
We will also write,
\begin{equation}\label{skldef}
s_{k,\ell}=\Tr \bL_1C_k\otimes C_\ell,\quad s_{k,\ell}^{(1)}=\Tr \bR\bL_1C_k\otimes C_\ell.
\end{equation}
With this notation, we see that
\begin{align}\label{Q0kformulas}
\bQ_0(0)&=
\begin{pmatrix}  \bL_2^* & 0\\
                                   \tilde{\bL}_1C_2\otimes A_{2,+}\bL_2^*   & \tilde{\bL}_1
                                              \end{pmatrix},
                                              \\
                                             \bQ_0(1)&=\begin{pmatrix}  0&  0\\
                              \tilde{\bL}_1C_1\otimes A_{2,+}\bL_2^*   &-\bR\tilde{\bL}_1 
                                              \end{pmatrix},\notag\\
                                              \bQ_0(2)&=
                                              \begin{pmatrix} 0& 0\\
                         \tilde{\bL}_1C_{12}\otimes A_{2,+}\bL_2^*   &\bR^2\tilde{\bL}_1 
                                              \end{pmatrix}.\notag
\end{align}
Using (\ref{Qrformulas}), Lemma \ref{LemkiMiformulas}, (\ref{Ckdef}) and (\ref{Bkdef}), we get
\begin{align}\label{Q1kformulas}
\bQ_1(1)&=
\begin{pmatrix}  B_7\otimes A_{2,+}& A_{2,-}\otimes C_4\\
                                C_3\otimes A_{2,+}^{(1)}& 0
                                              \end{pmatrix},
                                              \\
                                             \bQ_1(0)&=\begin{pmatrix}  B_2\otimes A_{2,+}&  A_{2,-}\otimes C_5\\
                              C_6\otimes A_{2,+}^{(1)}&0
                                              \end{pmatrix},\notag\\
                                            \bQ_1(-1)&=\begin{pmatrix}  B_1\otimes A_{2,+}&  A_{2,-}\otimes C_7\\
                                       0&0
                                              \end{pmatrix},\notag
\end{align}
and
\begin{align}\label{Q2kformulas}
\bQ_2(0)&=
\begin{pmatrix}  B_3\otimes A_{2,+}+B_4\otimes A_{2,+}^{(1)}& A_{2,-}^{(1)}\otimes C_8\\
                                C_9\otimes A_{2,+}^{(2)}& 0
                                              \end{pmatrix},
                                              \\
                                           \bQ_2(-1)&=
\begin{pmatrix}  B_5\otimes A_{2,+}+B_6\otimes A_{2,+}^{(1)}& A_{2,-}^{(1)}\otimes C_{10}\\
                               0& 0
                                              \end{pmatrix}.\notag
\end{align}

If we look at the formulas (\ref{e1formula}) and (\ref{e2formula}), we see that for $e_1$ we need to compute $\Tr \bP_1(-1)$ and $\Tr \bP_1(0)$. For $e_2$, we need to find
$\bP_1(-1)$, $\bP_1(0)$, $\bP_1(1)$, $\Tr\bP_2(-1)$, $\Tr\bP_2(0)$ and then $\sigma(-1,-1)$, $\sigma(-1,-1)$, $\sigma(-1,0)$, $\sigma(0,0)$ and $\sigma(-1,1)$ defined by
(\ref{sigmakl}). Note that (\ref{a0}) can be written
\begin{equation*}
a_0=\Tr\bL_2^*A_{2,-}\otimes A_{2,+},
\end{equation*}
and similarly for (\ref{a1}) and (\ref{a1star}). It follows from (\ref{Q0kformulas}), (\ref{Q1kformulas}) and (\ref{Q2kformulas}) that
\begin{align}\label{Qprodform}
\bQ_0(0)\bQ_1(-1)&=
\begin{pmatrix}  \bL_2^*B_1\otimes A_{2,+} & \bL_2^*A_{2,-}\otimes C_7\\
                                  a_0b_1\tilde{\bL}_1C_2\otimes A_{2,+}  & a_0\tilde{\bL}_1C_2\otimes C_7
                                              \end{pmatrix},\\
         \bQ_0(0)\bQ_1(0)&=     
  \begin{pmatrix}       \bL_2^*B_2\otimes A_{2,+} & \bL_2^*A_{2,-}\otimes C_5\\
                                  a_0b_2\tilde{\bL}_1C_2\otimes A_{2,+}+\tilde{\bL}_1C_6\otimes A_{2,+}^{(1)} & a_0\tilde{\bL}_1C_2\otimes C_5
                                              \end{pmatrix} ,\notag\\
                                         \bQ_0(1)\bQ_1(-1)&=
\begin{pmatrix}  0 & 0\\
                                  a_0b_1\tilde{\bL}_1C_{11}\otimes A_{2,+}  & a_0\tilde{\bL}_1C_{11}\otimes C_7
                                              \end{pmatrix},  \notag\\
                           \bQ_0(0)\bQ_1(1)&=    
                            \begin{pmatrix}            \bL_2^*B_7\otimes A_{2,+} & \bL_2^*A_{2,-}\otimes C_4\\
                                  a_0b_7\tilde{\bL}_1C_2\otimes A_{2,+}+\tilde{\bL}_1C_3\otimes A_{2,+}^{(1)} & a_0\tilde{\bL}_1C_2\otimes C_4
                                              \end{pmatrix} ,\notag\\        
                         \bQ_0(1)\bQ_1(0)&=
\begin{pmatrix}  0 & 0\\
                                  a_0b_2\tilde{\bL}_1C_{11}\otimes A_{2,+} -\bR\tilde{\bL}_1C_6\otimes A_{2,+}^{(1)} & a_0\tilde{\bL}_1C_{11}\otimes C_5
                                              \end{pmatrix},  \notag\\ 
                           \bQ_0(2)\bQ_1(-1)&=
\begin{pmatrix}  0 & 0\\
                                  a_0b_2\tilde{\bL}_1C_{12}\otimes A_{2,+}  & a_0\tilde{\bL}_1C_{12}\otimes C_7
                                              \end{pmatrix}. \notag         
\end{align}
From (\ref{Prform}), (\ref{Q0kformulas}) and (\ref{Q2kformulas}), we obtain
\begin{align}\label{P2trace1}
\Tr \bP_2(-1)&=\Tr\bL_2^*B_5\otimes A_{2,+}+\Tr\bL_2^*B_6\otimes A_{2,+}^{(1)}+a_1\Tr\tilde{\bL}_1C_2\otimes C_{10}\\
&=a_1(b_5+s_{2,10})+a_1^*b_6,\notag
\end{align}
and
\begin{align}\label{P2trace2}
\Tr \bP_2(0)&=\Tr\bL_2^*B_3\otimes A_{2,+}+\Tr\bL_2^*B_4\otimes A_{2,+}^{(1)}+a_1\Tr\tilde{\bL}_1C_2\otimes C_{8}+a_1\Tr\tilde{\bL}_1C_{11}\otimes C_{10}\\
&=a_1(b_3+s_{2,8}+s_{11,10})+a_1^*b_4.\notag
\end{align}

From (\ref{Prform}) and (\ref{Qprodform}), we get the following expressions
\begin{align}\label{P1formulas}
\bP_1(-1)=&\begin{pmatrix}       \bL_2^*B_1\otimes A_{2,+} & \bL_2^*A_{2,-}\otimes C_7\\
                                  a_0b_2\tilde{\bL}_1C_2\otimes A_{2,+}& a_0\tilde{\bL}_1C_2\otimes C_7
                                              \end{pmatrix},\\
         \bP_1(0)=    
&\begin{pmatrix}  \bL_2^*B_2\otimes A_{2,+} & \bL_2^*A_{2,-}\otimes C_5\\
                                  a_0(b_2\tilde{\bL}_1C_2+b_2\tilde{\bL}_1C_{11})\otimes A_{2,+} +\tilde{\bL}_1C_6\otimes A_{2,+}^{(1)} & a_0(\tilde{\bL}_1C_2\otimes C_5+
                                  \tilde{\bL}_1C_{11}\otimes C_7)
                                              \end{pmatrix},\notag\\   
                                                  \bP_1(1)=    
&\left(\begin{matrix}  \bL_2^*B_7\otimes A_{2,+}\\
                               a_0(b_7\tilde{\bL}_1C_2+b_2\tilde{\bL}_1C_{11}+b_1\tilde{\bL}_1C_{12})\otimes A_{2,+} +(\tilde{\bL}_1C_6-\bR\tilde{\bL}_1C_6)\otimes A_{2,+}^{(1)}   
                               \end{matrix}\right.\notag\\
                               &\left.\begin{matrix}
                                \bL_2^*A_{2,-}\otimes C_4\\ 
                                  a_0(\tilde{\bL}_1C_2\otimes C_4+
                                  \tilde{\bL}_1C_{11}\otimes C_5+\tilde{\bL}_1C_{12}\otimes C_7)
                                              \end{matrix}\right).     \notag
\end{align}
We can now use (\ref{e1formula}) to compute $e_1$. We see that
\begin{align}\label{P1trace}
\Tr \bP_1(-1)&=\Tr\bL_2^*B_1\otimes A_{2,+}+a_0\Tr\tilde{\bL}_1C_2\otimes C_{7}
=a_0(b_1+s_{2,7})\\
\Tr \bP_1(0)&=\Tr\bL_2^*B_2\otimes A_{2,+}+a_0(\Tr\tilde{\bL}_1C_2\otimes C_{5}+\Tr\tilde{\bL}_1C_{11}\otimes C_{7})\notag\\
&=a_0(b_2+s_{2,5}+s_{11,7}),\notag\\
\Tr \bP_1(1)&=\Tr\bL_2^*B_7\otimes A_{2,+}+a_0(\Tr\tilde{\bL}_1C_2\otimes C_{4}+\Tr\tilde{\bL}_1C_{11}\otimes C_{5}+\Tr\tilde{\bL}_1C_{12}\otimes C_{7})\notag\\
&=a_0(b_7+s_{2,4}+s_{11,5}+s_{12,7}),\notag
\end{align}
and thus
\begin{equation}\label{e1comp}
e_1=a_0[r_1(b_1+s_{2,7})+b_1+b_2+s_{2,5}+s_{2,7}+s_{11,7}].
\end{equation}
Now, by (\ref{Ckdef}),(\ref{grs}) and (\ref{skldef}),
\begin{align*}
s_{2,7}&=\Tr (\bI-\bK_1)^{-1}(\bK_1^{(-1)}-(\bK_1A_{1,+})^{(-1)})\otimes (A_{1,-}^{(-1)}-(A_{1,-}^0)^{(-1)})\\
&=g_1(1,0)-g_1(0,0)-g_1(0,0)+g_1(0,1),
\end{align*}
and
\begin{align*}
s_{2,5}+s_{2,7}&=\Tr\tilde{\bL}_1C_2\otimes(C_5+C_7)=
\Tr (\bI-\bK_1)^{-1}(\bK_1^{(-1)}-(\bK_1A_{1,+})^{(-1)})\otimes (-A_{1.-}^{(-1)})\\
&=g_1(1,1)-g_1(1,0).
\end{align*}
All computations of $s_{k,\ell}$ are done in an analogous manner and below we will not give the details in each case. We find,
\begin{equation*}
s_{11,7}=g_1(1,1)-g_1(0,1)+g_2(1,1)-g_2(1,0)-g_2(0,1)+g_2(0,0).
\end{equation*}
Using (\ref{grsfor}) and (\ref{bkdef}), we get
\begin{equation}\label{e1comp2}
b_1+b_2+s_{2,5}+s_{2,7}+s_{11,7}=
b_2(1,1)+b_1(1,0)-b_2(1,0)+b_1(0,1)-b_2(0,1)-b_1(0,0)+b_2(0,0),
\end{equation}
and
\begin{equation}\label{e1comp3}
b_1+b_2+s_{2,7}=
\xi_1-b_1(1,1)-b_1(0,0)+b_1(1,0)+b_1(0,1).
\end{equation}
Combining (\ref{e1comp}), (\ref{e1comp2}) and (\ref{e1comp3}), we arrive at (\ref{e1}) with $\psi_1$ and $\psi_2$ given by (\ref{g1}).

Next, we turn to the computation of $e_2$. First, we will compute $\sigma(-1,-1)$, $\sigma(0,0)$, $\sigma(-1,0)$ and $\sigma(-1,1)$. From (\ref{P1formulas}) and (\ref{P1trace}), we see that
\begin{align*}
\Tr\bP_1(-1)^2&=\Tr(\bL_2^*B_1\otimes A_{2,+})(\bL_2^*B_1\otimes A_{2,+})+2a_0b_1(\Tr\tilde{\bL}_1C_2\otimes A_{2,+})(\bL_2^*A_{2,-}\otimes C_7)\\&+
a_0^2(\Tr\tilde{\bL}_1C_2\otimes C_7)(\tilde{\bL}_1C_2\otimes C_7)\\
&=a_0^2(b_1^2+2b_1s_{2,7}+s_{2,7}^2)=(\Tr\bP_1(-1))^2.
\end{align*}
Thus,
\begin{equation}\label{sigmaminus11}
\sigma(-1,-1)=0.
\end{equation}
Similar computations give
\begin{equation}\label{sigma00}
\sigma(0,0)=2a_1^*s_{6,5}+2a_0^2(b_1s_{11,5}-b_2s_{11,7}+s_{11,5}s_{2,7}-s_{2,5}s_{11,7}),
\end{equation}
\begin{equation}\label{sigmaminus10}
\sigma(-1,0)=a_1^*s_{6,7},
\end{equation}
and
\begin{equation}\label{sigmaminus1plus1}
\sigma(-1,1)=a_1^*(s_{3,7}-s_{6,7}^{(1)})+a_0^2(b_2s_{11,7}-b_1s_{11,5}+s_{11,7}s_{2,5}-s_{2,7}s_{11,5}).
\end{equation}
It follows from (\ref{e2formula}) and (\ref{sigmaminus11}) that
\begin{equation}\label{e2finalformula}
e_2=r_1\left(\frac 12\Tr\bP_2(-1)-\sigma(-1,0)\right)+\frac 12\Tr\bP_2(-1)+\frac 12\Tr\bP_2(0)-\sigma(-1,0)-\frac 12\sigma(0,0)-\sigma(-1,1).
\end{equation}
From (\ref{P2trace1}) and (\ref{sigmaminus10}), we obtain
\begin{equation}\label{e2part1}
\frac 12\Tr\bP_2(-1)-\sigma(-1,0)=\frac 12a_1(b_5+s_{2,10})+a_1^*(\frac 12b_6-s_{6,7}).
\end{equation}
We can now compute $s_{2,10}$ and $s_{6,7}$, which gives
\begin{align*}
\frac 12s_{2,10}&=b_{0,1}^*(1,1)-b_{1,1}^*(1,1)+\xi_1b_0(1,1)-\xi_1b_1(1,1)-b_{0,1}^*(1,0)+b_{1,1}(1,0)\\
&-\xi_1b_0(1,0)+\xi_1b_1(1,0)+b_{0,1}^*(0,1)-b_{1,1}^*(0,1)+b_{1,1}^*(0,0),
\end{align*}
and
\begin{align*}
-s_{6,7}&=b_{0,1}(1,1)-b_{1,1}(1,1)+\xi_1b_0(1,1)-\xi_1b_1(1,1)+b_{0,1}(1,0)-b_{1,1}(1,0)\\
&-b_{0,1}(0,1)+b_{1,1}(0,1)-\xi_1b_0(0,1)+\xi_1b_1(0,1)+b_{1,1}(0,0).
\end{align*}
Using, (\ref{phi1}), (\ref{bkdef}) and Lemma \ref{Lembrsrelation}, we see that
\begin{equation}\label{phicomp1}
\frac 12b_5+\frac 12s_{2,10}=\phi_1(\xi_1,-\eta_1),\quad \frac 12b_6-s_{6,7}=\phi_1(\xi_1,\eta_1).
\end{equation}
We see from (\ref{P2trace1}), (\ref{P2trace2}), (\ref{sigmaminus10}), (\ref{sigma00}) and (\ref{sigmaminus1plus1}) that
\begin{align}\label{e2part2}
&\frac 12\Tr\bP_2(-1)+\frac 12\Tr\bP_2(0)-\sigma(-1,0)-\frac 12\sigma(0,0)-\sigma(-1,1)\\
&=\frac {a_1}2 \left(b_3+b_5+s_{2,10}+s_{2,8}+s_{11,10}\right)
+a_1^*\left(\frac 12b_4+\frac 12 b_6-s_{6,7}-s_{6,5}-s_{3,7}+s_{6,7}^{(1)}\right).\notag
\end{align}
A computation gives
\begin{align*}
\frac 12(s_{2,8}+s_{2,10}+s_{11,10})&=-b_{0,1}^*(1,1)+b_{2,1}^*(1,1)-\xi_1b_0(1,1)+\xi_1b_2(1,1)+b_{1,1}^*(1,0)-b_{2,1}^*(1,0)\\
&+\xi_1b_1(1,0)-\xi_1b_2(1,0)-b_{1,1}(0,1)+b_{2,1}^*(0,1)+b_{1,1}^*(0,0)-b_{2,1}^*(0,0).
\end{align*}
By (\ref{bkdef}),
\begin{equation*}
\frac 12(b_3+b_5)=b_{0,1}^*(1,1)+\xi_1b_0(1,1)+b_{0,1}(1,0)-b_{0,1}(0,1),
\end{equation*}
and using Lemma \ref{Lembrsrelation} and (\ref{phi2}), we obtain, 
\begin{equation}\label{phicomp2}
b_3+b_5+s_{2,10}+s_{2,8}+s_{11,10}=\phi_2(\xi_1,-\eta_1).
\end{equation}
from (\ref{phi2}). Next, a computation gives
\begin{align*}
-s_{6,7}-s_{6,5}-s_{3,7}+s_{6,7}^{(1)}&=-b_{0,1}(1,1)+b_{2,1}(1,1)-\xi_1b_0(1,1)+\xi_1b_2(1,1)-b_{1,1}(1,0)+b_{2,1}(1,0)\\
&+b_{1,1}(0,1)-b_{2,1}(0,1)+\xi_1b_1(0,1)-\xi_1b_2(0,1)+b_{1,1}(0,0)-b_{2,1}(0,0),
\end{align*}
and from (\ref{bkdef}), we find
\begin{equation*}
\frac 12(b_4+b_6)=b_{0,1}(1,1)+\xi_1b_0(1,1)-b_{0,1}(1,0)+b_{0,1}(0,1).
\end{equation*}
Thus, by (\ref{phi2}),
\begin{equation}\label{phicomp3}
\frac 12b_4+\frac 12 b_6-s_{6,7}-s_{6,5}-s_{3,7}+s_{6,7}^{(1)}=\phi_2(\xi_1,\eta_1).
\end{equation}
This completes the proof of Theorem \ref{Thmalphazero}.
\end{proof}

\section{Proof of the short time expansion}\label{secThmalphainfty}
We turn now to the proof of Theorem \ref{Thmalphainfty}. Note that if we insert (\ref{xi2eta2}) into (\ref{deltaxideltaeta}), then $\Delta\xi=\xi$ and $\Delta\eta=\eta$. Let
$\bQ=\bQ(u,\alpha,\xi_1,\Delta\eta,\eta_1,\Delta\eta,\delta)$ be the kernel in (\ref{Qformula1}). From (6.26) in \cite{Jott}, we have the formula
\begin{equation}\label{Fttnew}
F_{tt}(\xi_1,\eta_1;\xi_2,\eta_2;\alpha)=\frac 1{\pii}\int_{\gamma_r}\det(\bI+\bQ(u^{-1},\beta,\Delta\xi,\xi_1,\Delta\eta,\eta_1,\delta))_{Y}\frac{du}{u-1},
\end{equation}
where $r>1$ and
\begin{equation}\label{beta}
\beta=\frac 1{\alpha}.
\end{equation}
We will write
\begin{equation}\label{betaprime}
\beta'=(1+\beta^3)^{1/3}.
\end{equation}
Note that $\xi_1$ and $\Delta\xi$, and also $\eta_1$, and $\Delta\eta$ have changed places in (\ref{Fttnew}). The formulas for $\bQ$ contain $\xi_2$ and $\eta_2$ and they
should be thought of as functions of $\xi_1, \Delta\xi$ and $\eta_1,\Delta\eta$ respectively using (\ref{deltaxideltaeta}), i.e.
\begin{equation*}
\xi_2=\frac 1{\alpha'}(\Delta\xi+\alpha\xi_1),\quad \eta_2=\frac 1{\alpha'^2}(\Delta\eta+\alpha^2\eta_1).
\end{equation*}
Deform the contour $\gamma_r$ in (\ref{Fttnew}) to $\gamma_{1/r}$. This gives
\begin{align}\label{Fttnew2}
F_{tt}(\xi_1,\eta_1;\xi_2,\eta_2;\alpha)&=\det(\bI+\bQ(1,\beta,\Delta\xi,\xi_1,\Delta\eta,\eta_1,\delta))_Y\\
&+\frac 1{\pii}\int_{\gamma_{1/r}}\det(\bI+\bQ(u^{-1},\beta,\Delta\xi,\xi_1,\Delta\eta,\eta_1,\delta))_{Y}\frac{du}{u-1}.\notag
\end{align}
We have the following Lemma that will be proved in section \ref{secLemmas}.
\begin{lemma}\label{LemQF2}
We have the formula,
\begin{equation}\label{QF2for}
\det(\bI+\bQ(1,\beta,\Delta\xi,\xi_1,\Delta\eta,\eta_1,\delta))_Y=F_2(\xi_2).
\end{equation}
\end{lemma}

The change of variables $u\to 1/u$ in (\ref{Fttnew2}) gives
\begin{equation*}
F_{tt}(\xi_1,\eta_1;\xi_2,\eta_2;\alpha)=F_2(\xi_2)-\frac 1{\pii}\int_{\gamma_r}\det(\bI+\bQ(u,\beta,\Delta\xi,\xi_1,\Delta\eta,\eta_1,\delta))_{Y}\frac{du}{u(u-1)}.
\end{equation*}
Thus, if we write
\begin{equation*}
\tilde{\bQ}=\bQ(u,\beta,\Delta\xi,\xi_1,\Delta\eta,\eta_1,\delta),
\end{equation*}
we see that
\begin{equation}\label{dxi1Ftt}
\frac{\partial F_{tt}}{\partial\xi_1}(\xi_1,\eta_1;\xi_2,\eta_2;\alpha)=-\frac 1{\pii}\int_{\gamma_r}\frac{d}{d\xi_1}\det(\bI+\tilde{\bQ})\frac{du}{u(u-1)}.
\end{equation}
We write the total derivative w.r.t. $\xi_1$ since $\tilde{\bQ}$ depends on $\xi_1$ directly and through $\Delta\xi$,
\begin{equation*}
\frac{\partial}{\partial\xi_1}\Delta\xi=-\frac 1{\beta}.
\end{equation*}
Now,
\begin{align}\label{dxi1detQ}
\frac{d}{d\xi_1}\det(\bI+\tilde{\bQ})&=\det(\bI+\tilde{\bQ})\Tr(\bI+\tilde{\bQ})^{-1}\frac{d\tilde{\bQ}}{d\xi_1}\\
&=\det(\bI+\tilde{\bQ})\Tr(\bI+\tilde{\bQ})^{-1}\frac{\partial\tilde{\bQ}}{\partial\xi_1}-\frac 1{\beta}
\det(\bI+\tilde{\bQ})\Tr(\bI+\tilde{\bQ})^{-1}\frac{\partial\tilde{\bQ}}{\partial\Delta\xi}.\notag
\end{align}
In this formula, we should substitute $\xi_2$ and $\eta_2$ given by
\begin{equation}\label{xi2eta22}
\xi_2=\frac 1{\beta'}(\xi_1+\beta\xi),\quad \eta_2=\frac 1{\beta'^2}(\eta_1+\beta^2\eta).
\end{equation}
Note that if we insert (\ref{xi2eta22}) into (\ref{deltaxideltaeta}), then $\Delta\xi=\xi$ and $\Delta\eta=\eta$.
Write,
\begin{equation}\label{Pdef}
\bP=\bQ(u,\beta,\xi,\xi_1,\eta,\eta_1,\delta).
\end{equation}
After this substitution, which replaces $\Delta\xi$ with $\xi$ and $\Delta\eta$ with $\eta$, which are both fixed, we see that
\begin{align}\label{dQdP}
\frac{\partial\tilde{\bQ}}{\partial\xi_1}(u,\beta,\xi,\xi_1,\eta,\eta_1,\delta)&=\frac{\partial\bP}{\partial\xi_1},\\
\frac{\partial\tilde{\bQ}}{\partial\Delta\xi_1}(u,\beta,\xi,\xi_1,\eta,\eta_1,\delta)&=\frac{\partial\bP}{\partial\xi}.\notag
\end{align}
Hence, by (\ref{dxi1Ftt}) to (\ref{dQdP}), we get
\begin{align}\label{dxi1Ftt2}
&\frac{\partial F_{tt}}{\partial\xi_1}(\xi_1,\eta_1;\frac{\xi+\alpha\xi_1}{\alpha'},\frac{\eta+\alpha^2\xi_1}{\alpha'^2};\alpha)=
-\frac 1{\pii}\int_{\gamma_r}\det(\bI+\bP)\Tr(\bI+\bP)^{-1}\frac{\partial\bP}{\partial\xi_1}\frac{du}{u(u-1)}\\
&+\frac 1{\pii\beta}\int_{\gamma_r}\det(\bI+\bP)\Tr(\bI+\bP)^{-1}\frac{\partial\bP}{\partial\xi}\frac{du}{u(u-1)}\notag\\
&=-\frac{\partial}{\partial\xi_1}\frac 1{\pii}\int_{\gamma_r}\det(\bI+\bP)\frac{du}{u(u-1)}
+\frac 1{\beta}\frac{\partial}{\partial\xi}\frac 1{\pii}\int_{\gamma_r}\det(\bI+\bP)\frac{du}{u(u-1)}.\notag
\end{align}
Write
\begin{equation}\label{Gdef}
G=G(\beta)=G(\beta,\xi_1,\eta_1,\xi,\eta)=\frac{\partial F_{tt}}{\partial\xi_1}(\xi_1,\eta_1;\frac{\xi+\alpha\xi_1}{\alpha'},\frac{\eta+\alpha^2\xi_1}{\alpha'^2};\alpha),
\end{equation}
and
\begin{equation}\label{Hdef}
H=H(\beta)=H(\beta,\xi_1,\eta_1,\xi,\eta)=\frac 1{\pii}\int_{\gamma_r}\det(\bI+\bP)\frac{du}{u(u-1)}.
\end{equation}
It follows from (\ref{dxi1Ftt2}), (\ref{Gdef}) and (\ref{Hdef}) that
\begin{equation}\label{Gformula}
G=-\frac{\partial H}{\partial\xi_1}+\frac 1{\beta}\frac{\partial H}{\partial\xi}.
\end{equation}
We can expand in powers of $\beta$,
\begin{equation}\label{Hexp}
H(\beta)=H_0+H_1\beta+\frac 12H_2\beta^2+\dots,
\end{equation}
where
\begin{equation}\label{Hs}
H_s=\left.\frac{\partial^s}{\partial\beta^s}\right|_{\beta=0}H.
\end{equation}
Also, we expand
\begin{equation}\label{Pexp}
\bP=\bP_0+\bP_1\beta+\frac 12\bP_2\beta^2+\dots.
\end{equation}
To proceed we need to get a formula for $\bP_s$.

If we replace $\xi_1$ with $\xi$, $\Delta\xi$ with $\xi_1$, $\eta_1$ with $\eta$, $\Delta\eta$ with $\eta_1$, $\alpha$ with $\beta$ and $\alpha'$ with $\beta'$ in
(\ref{M1}) to (\ref{M3}) and in (\ref{k1}) to (\ref{k7}), we obtain
\begin{equation}\label{M1tilde}
\tilde{\bM}_1(v_1,v_2)=\frac{e^{\delta(v_1-v_2)}}{(\pii)^2}\int_{\Gamma_D}dz\int_{\Gamma_{-d}}d\zeta\frac{G_{\xi+v_1,\eta}(z)}{G_{\xi+v_2,\eta}(\zeta)(z-\zeta)}
\end{equation}
\begin{equation}\label{M2tilde}
\tilde{\bM}_2(v_1,v_2)=\frac{1}{(\pii)^2\beta'}\int_{\Gamma_D}dz\int_{\Gamma_{-d}}d\zeta\frac{G_{\xi_2+v_2/\beta',\eta_2}(z)}{G_{\xi_2+v_1/\beta',\eta_2}(\zeta)(z-\zeta)}
\end{equation}
\begin{equation}\label{M3tilde}
\tilde{\bM}_3(v_1,v_2)=\frac{1}{(\pii)^2}\int_{\Gamma_D}dz\int_{\Gamma_{-d}}d\zeta\frac{G_{\xi_1+v_2,\eta_1}(z)}{G_{\xi_1+v_1,\eta_1}(\zeta)(z-\zeta)},
\end{equation}
and
\begin{align}\label{k1tilde}
&\tilde{\bk}_1(v_1,v_2)\\&=\frac{\beta}{(\pii)^4}\int_{\Gamma_{D_3}}dz\int_{\Gamma_{D_2}}dw\int_{\Gamma_{-d_3}}d\zeta\int_{\Gamma_{-d_2}}d\omega
\frac{G_{\xi,\eta}(z)G_{\xi_1+v_2,\eta_1}(w)}{G_{\xi,\eta}(\zeta)G_{\xi_1+v_1,\eta_1}(\omega)(z-\zeta)(z-\beta w)(\beta\omega-\zeta)},\notag
\end{align}
\begin{equation}\label{k2tilde}
\tilde{\bk}_2(v_1,v_2)=\frac{\beta}{(\pii)^3}\int_{\Gamma_{D_3}}dz\int_{\Gamma_{D_2}}dw\int_{\Gamma_{-d_2}}d\omega
\frac{G_{\xi,\eta}(z)G_{\xi_1+v_2,\eta_1}(w)}{G_{\xi_2+v_1/\beta',\eta_2}(\omega)(\beta'z-\beta\omega)(z-\beta w)},
\end{equation}
\begin{equation}\label{k3tilde}
\tilde{\bk}_3(v_1,v_2)=\frac{\beta e^{-\delta v_2}}{(\pii)^2}\int_{\Gamma_{-d_3}}d\zeta\int_{\Gamma_{-d_2}}d\omega
\frac 1{G_{\xi+v_2,\eta}(\zeta)G_{\xi_1+v_1,\eta_1}(\omega)(\beta\omega-\zeta)},
\end{equation}
\begin{equation}\label{k4tilde}
\tilde{\bk}_4(v_1,v_2)=\frac{\beta e^{-\delta v_2}}{\beta'\pii}\int_{\Gamma_{-d_2}}\frac{d\omega}{G_{\xi_2+(v_1+\beta v_2)/\beta',\eta_2}(\omega)},
\end{equation}
\begin{equation}\label{k5tilde}
\tilde{\bk}_5(v_1,v_2)=\frac{\beta}{(\pii)^3}\int_{\Gamma_{D_2}}dw\int_{\Gamma_{-d_3}}d\zeta\int_{\Gamma_{-d_2}}d\omega
\frac{G_{\xi_2+v_2/\alpha',\eta_2}(w)}{G_{\xi,\eta}(\zeta)G_{\xi_1+v_1,\eta_1}(\omega)(\beta w-\beta'\zeta)(\beta\omega-\zeta)},
\end{equation}
\begin{equation}\label{k6tilde}
\tilde{\bk}_6(v_1,v_2)=\frac{e^{\delta v_1}}{(\pii)^4}\int_{\Gamma_{D_3}}dz_1\int_{\Gamma_{D_1}}dz_2\int_{\Gamma_{D_2}}dw\int_{\Gamma_{-d_1}}d\zeta
\frac{G_{\xi,\eta}(z_1)G_{\xi+v_1,\eta}(z_2)G_{\xi_1+v_2,\eta_1}(w)}{G_{\xi,\eta}(\zeta)(z_1-\zeta)(z_2-\zeta)(z_1-\beta w)},
\end{equation}
and
\begin{equation}\label{k7tilde}
\tilde{\bk}_7(v_1,v_2)=\frac{e^{\delta v_1}}{(\pii)^3}\int_{\Gamma_{D_1}}dz\int_{\Gamma_{D_2}}dw\int_{\Gamma_{-d_1}}d\zeta
\frac{G_{\xi+v_1,\eta}(z)G_{\xi_2+v_2/\beta',\eta_2}(w)}{G_{\xi,\eta}(\zeta)(\beta w-\beta'\zeta)(z-\zeta)},
\end{equation}
where we still have the condition (\ref{dDconditions}). In these formulas, $\xi_2$ and $\eta_2$ are given by
(\ref{xi2eta2}),
so in particular,
\begin{equation}\label{dxi2}
\left.\frac{\partial\xi_2}{\partial\beta}\right|_{\beta=0}=\xi.
\end{equation}
By (\ref{Pdef}) and (\ref{Qformula1}), we obtain
\begin{equation}\label{Pformula}
\bQ(u)=\begin{pmatrix} (2-u-u^{-1})\tilde{\bk}_1+(u-1)(\tilde{\bk}_2+\tilde{\bk}_5)+(u-1)\tilde{\bM}_3-u\tilde{\bM}_2 & (u+u^{-1}-2)\tilde{\bk}_3+(1-u)\tilde{\bk}_4  \\
                                    (1-u^{-1})\tilde{\bk}_6-\tilde{\bk}_7          & (u^{-1}-1)\tilde{\bM}_1
                                              
                    \end{pmatrix}.
 \end{equation}
 We can expand the operators $\tilde{\bM}_j$ and $\tilde{\bk}_j$ in powers of $\beta$,
 \begin{align*}
\tilde{\bM}_j&=\tilde{\bM}_{j,0}+ \tilde{\bM}_{j,1}\beta+\dots\\
\tilde{\bk}_j&=\tilde{\bk}_{j,0}+ \tilde{\bk}_{j,1}\beta+\dots.
\end{align*}
Then, using (\ref{Pexp}) and (\ref{Pformula}), we see that
\begin{align}\label{Psfor}
\bP_s&=
u\begin{pmatrix} -\tilde{\bk}_{1,s}+\tilde{\bk}_{2,s}+\tilde{\bk}_{5,s}+\tilde{\bM}_{3,s}-\tilde{\bM}_{2,s}& \tilde{\bk}_{3,s}-\tilde{\bk}_{4,s}  \\
                                   0         & 0
                                               \end{pmatrix}\\
&+\begin{pmatrix} 2\tilde{\bk}_{1,s}-\tilde{\bk}_{2,s}-\tilde{\bk}_{5,s}-\tilde{\bM}_{3,s}& -2\tilde{\bk}_{3,s}+\tilde{\bk}_{4,s}  \\
                                \tilde{\bk}_{6,s}-\tilde{\bk}_{7,s}          & -\tilde{\bM}_{1,s}
                                              \end{pmatrix}+
                                           u^{-1}\begin{pmatrix} -\tilde{\bk}_{1,s}-& \tilde{\bk}_{3,s}  \\
                                -\tilde{\bk}_{6,s}         & \tilde{\bM}_{1,s}
                                              \end{pmatrix}  \notag\\
                                              &:=u\bP_s(-1)+\bP_s(0)+u^{-1}\bP_s(1).\notag
                                              \end{align} 
Recall (\ref{A0}) and (\ref{K0}). In analogy with (\ref{M1delta}), we define the conjugated kernel,
\begin{equation}\label{M0}
\bM_0(v_1,v_2)=c_{-\delta}(v_1)K_0(v_1,v_2)c_{\delta}(v_2)
\end{equation}
and like (\ref{a0}), we write
\begin{equation}\label{a0tilde}
\tilde{a}_0=\Tr(\bI-\bK_1^*)^{-1}A_{1,-}\otimes A_{1,+}=\frac{F_2'(\xi_1+\eta_1^2)}{F_2(\xi_1+\eta_1^2)},
\end{equation}
where the last equality is (\ref{a0TW}).
Recall (\ref{brtilde}) and (\ref{brtilde0}). Define, for $r\ge 1$,
\begin{equation}\label{grtilde}
\tilde{g}_r(\epsilon_1,\epsilon_2)=\Tr(\bI-\bK_0)^{-r}\bK_0^{r-1}(\bK_0A_{0,+}^{\epsilon_2})^{(-1)}\otimes(A_{0,-}^{\epsilon_1})^{(-1)}.
\end{equation}
Then, as in Lemma \ref{Lemgrsformula},
\begin{align}\label{g1tildefor}
\tilde{g}_1(\epsilon_1,\epsilon_2)&=-\tilde{b}_0(\epsilon_1,\epsilon_2)+\tilde{b}_1(\epsilon_1,\epsilon_2), \quad (\epsilon_1,\epsilon_2)\neq(0,0),\\
\tilde{g}_1(0,0)&=-\tilde{b}_0(0,0).\notag
\end{align}
\begin{lemma}\label{LemkiMitilde}
We have the following formulas
\begin{align*}
\tilde{\bk}_{1,0}&=\tilde{\bk}_{2,0}=\tilde{\bk}_{3,0}=\tilde{\bk}_{4,0}=\tilde{\bk}_{5,0}=0,
\\
\tilde{\bk}_{6,0}&=c_{-\delta}(\bK_0A_{0,+})^{(-1)}\otimes A_{1,+},\quad  \tilde{\bk}_{7,0}=c_{-\delta}(\bK_0)^{(-1)}\otimes A_{1,+},
\\
\tilde{\bM}_{1,0}&=\bM_0,\quad \tilde{\bM}_{1,0}=\tilde{\bM}_{1,0}=\bK_1^*,
\\
\tilde{\bk}_{1,1}&=\tilde{b}_0(0,0)A_{1,-}\otimes A_{1,+},\quad \tilde{\bk}_{2,1}=\tilde{b}_0(0,1)A_{1,-}\otimes A_{1,+}\quad \tilde{\bk}_{5,1}=\tilde{b}_0(1,0)A_{1,-}\otimes A_{1,+}
\\
\tilde{\bk}_{3,1}&=A_{1,-}\otimes A_{0,-}^{(-1)}c_\delta,\quad \tilde{\bk}_{4,1}=A_{1,-}\otimes c_\delta
\\
\tilde{\bM}_{1,1}&=\tilde{\bM}_{3,1}=0,\quad\tilde{\bM}_{2,1}=-\xi A_{1,-}\otimes A._{1,+}
\end{align*}
\end{lemma}

\begin{proof}
The proof is completely analogous to that of Lemma \ref{LemkiMiformulas} starting from (\ref{M1tilde}) to (\ref{k7tilde}).
\end{proof}
We will use the notation 
\begin{equation}\label{Ltilde0}
\tilde{\bL}_0=(\bI-\bM_0)^{-1}.
\end{equation}
It follows from (\ref{Psfor}) and Lemma \ref{LemkiMitilde}, that
\begin{equation}\label{P0for}
\bP_0=\begin{pmatrix} -\bK_{1}^* & 0  \\
                             (1-u^{-1})\tilde{\bk}_{6,0}- \tilde{\bk}_{7,0} & (u^{-1}-1)\tilde{\bM}_{0}.
                                              \end{pmatrix} 
             \end{equation}        
\begin{lemma}\label{LemP0exp}
For $|u|$ large enough
\begin{equation}\label{detP0}
\det(\bI+\bP_0)=F_2(\xi_1+\eta_1^2)F_2(\xi+\eta^2)\left(1+\sum_{k=1}^\infty\tilde{r}_k u^{-k}\right),
\end{equation}
and
\begin{equation}\label{IdP0inv}
(\bI+\bP_0)^{-1}=\sum_{k=0}^\infty \bJ(k)u^{-k}.
\end{equation}                   
In particular
\begin{equation}\label{J0}
\bJ(0)=\begin{pmatrix} -\bL_{1}^* & 0  \\
                             \tilde{\bL}_0\tilde{C}_1\otimes A_{1,+}\bL_1^* & \tilde{\bL}_0
                                              \end{pmatrix} ,
             \end{equation}  
where
\begin{equation}\label{C1tilde}
\tilde{C}_1=c_{-\delta}(\bK_0^{(-1)}-(\bK_0A_{0,+})^{(-1)}).
\end{equation}   
\end{lemma}

\begin{proof} 
This is completely analogous to the proof of Lemma \ref{LemQ0exp}. We have also used Lemma \ref{LemkiMitilde}.
\end{proof}

From (\ref{detP0}) and the fact that
\begin{equation}\label{uint2}
\frac 1{\pii}\int_{\gamma_r}\frac{u^k}{u(u-1)}du=\begin{cases} 1 & \text{if $k\ge 1$} \\
0 & \text{if $k\le 0$},
\end{cases}
\end{equation}
since $r>1$, we see that
\begin{equation}\label{H0zero}
H_0=\frac 1{\pii}\int_{\gamma_r}\det(\bI+\bP_0)\frac{du}{u(u-1)}=0.
\end{equation}                     
Thus, by (\ref{Hexp}),
\begin{equation}\label{Hexp2}
H_1(\beta)=H_1\beta+\frac 12H_2\beta^2+\dots.
\end{equation}    
Using analyticity we get
\begin{align}\label{dH}
\frac{\partial H}{\partial\xi_1}&=\frac{\partial H_1}{\partial\xi_1}\beta+\frac 12\frac{\partial H_2}{\partial\xi_1}\beta^2+\dots,\\
\frac{\partial H}{\partial\xi}&=\frac{\partial H_1}{\partial\xi}\beta+\frac 12\frac{\partial H_2}{\partial\xi}\beta^2+\dots,\notag
\end{align}
and inserting this into (\ref{Gformula}) we find
\begin{equation}\label{GHfor}
G(\beta)=\frac{\partial H_1}{\partial\xi}+\left(\frac 12\frac{\partial H_2}{\partial\xi}-\frac{\partial H_1}{\partial\xi_1}\right)\beta+\dots.
\end{equation}  
From (\ref{Hdef}) and (\ref{Hs}) we obtain
 \begin{equation}\label{H1for}
H_1=\frac 1{\pii}\int_{\gamma_r}\det(\bI+\bP_0)\Tr(\bI+\bP_0)^{-1}\bP_1\frac{du}{u(u-1)}.
\end{equation}   
We will only compute $G(0)$, i.e.
\begin{equation}\label{G0for}
G(0)=\frac{\partial H_1}{\partial\xi}.
\end{equation}  
The next term in the expansion in (\ref{GHfor}) can also be computed with more effort. 

\begin{proof}[Proof of Theorem \ref{Thmalphainfty}]
If we insert (\ref{detP0}), (\ref{IdP0inv}) into (\ref{G0for}) and use
\begin{equation*}
\bP_1=u\bP_1(-1)+\bP_1(0)+u^{-1}\bP_1(1),
\end{equation*}
from (\ref{Psfor}), we see that
\begin{align}\label{H1for2}
H_1&=F_2(\xi_1+\eta_1^2)F_2(\xi+\eta^2)\frac 1{\pii}\int_{\gamma_r}\left(1-\sum_{k=1}^\infty\tilde{r}_k u^{-k}\right)\\
&\times\Tr\left[\left(\sum_{k=0}^\infty \bJ(k)u^{-k}\right)
\left(u\bP_1(-1)+\bP_1(0)+u^{-1}\bP_1(1)\right)\right]\frac{du}{u(u-1)}.\notag
\end{align}
Using (\ref{uint2}), it follows that
\begin{equation}\label{H1for4}
H_1=F_2(\xi_1+\eta_1^2)F_2(\xi+\eta^2)\Tr\bJ(0)\bP_1(-1).
\end{equation} 
From Lemma \ref{LemkiMitilde}, we obtain
\begin{align*}
 -\tilde{\bk}_{1,1}+\tilde{\bk}_{2,1}+\tilde{\bk}_{5,1}+\tilde{\bM}_{3,1}-\tilde{\bM}_{2,1}&=(-\tilde{b}_0(0,0)+\tilde{b}_0(0,1)+\tilde{b}_0(1,0)+\xi)A_{1,-}\otimes A_{1,+}\\
 &:=d_1A_{1,-}\otimes A_{1,+},
 \end{align*}
 and
 \begin{equation*}
  \tilde{\bk}_{3,1}-\tilde{\bk}_{4,1}=A_{1,-}\otimes(A_{0,-}^{(-1)}-1)c_\delta:=A_{1,-}\otimes\tilde{C}_2.
 \end{equation*}
 Thus, by (\ref{Psfor}),
 \begin{equation*}
 \bP_1(0)=\begin{pmatrix} d_1A_{1,-}\otimes A_{1,+} & A_{1,-}\otimes\tilde{C}_2  \\
                             0& 0
                                              \end{pmatrix} ,
             \end{equation*}
 Combined with (\ref{J0}), this gives
 \begin{align*}
 \Tr\bJ(0)\bP_1(-1)&=\Tr\begin{pmatrix} -\bL_{1}^* & 0  \\
                             \tilde{\bL}_0\tilde{C}_1\otimes A_{1,+}\bL_1^* & \tilde{\bL}_0
                                              \end{pmatrix}\begin{pmatrix} d_1A_{1,-}\otimes A_{1,+} & A_{1,-}\otimes\tilde{C}_2  \\
                             0& 0
                                              \end{pmatrix}\\
    &=d_1\Tr\bL_1^*A_{1,-}\otimes A_{1,+}+(\Tr \bL_1^*A_{1,-}\otimes A_{1,+})(\Tr\tilde{\bL}_0\tilde{C}_1\otimes\tilde{C}_2)=\tilde{a}_0\left(d_1
    +\Tr\tilde{\bL}_0\tilde{C}_1\otimes\tilde{C}_2\right).
    \end{align*}
Now,
\begin{align*}
\Tr\tilde{\bL}_0\tilde{C}_1\otimes\tilde{C}_2&=\Tr(\bI-\bK_0)^{-1}(\bK_0^{(-1)}-(\bK_0A_{0,+})^{(-1)})\otimes(A_{0,-}^{(-1)}-(A_{0,-}^0)^{(-1)})  \\
&=\tilde{g}_1(1,0)-\tilde{g}_1(0,0)-\tilde{g}_1(1,1)+\tilde{g}_1(0,1)  \\
&=\tilde{b}_0(1,1)-\tilde{b}_0(1,0) -\tilde{b}_0(0,1)-\tilde{b}_1(0,0)  -\tilde{b}_1(1,1)+\tilde{b}_1(1,0)  +\tilde{b}_1(0,1),
\end{align*}      
by (\ref{grtilde}) and (\ref{g1tildefor}). Since $\tilde{b}_0(1,1)=\tilde{b}_0(0,0)$, we see that
\begin{equation*}
\Tr\bJ(0)\bP_1(-1)=\frac{F_2'(\xi_1+\eta_1^2)}{F_2(\xi_1+\eta_1^2)}\left(\xi-\tilde{b}_1(0,0)-\tilde{b}_1(1,1)+\tilde{b}_1(0,1)+\tilde{b}_1(1,0)\right).
\end{equation*}
and hence
\begin{equation}\label{H1for3}
H_1= F_2'(\xi_1+\eta_1^2)F_2(\xi+\eta^2)\left(\xi-\tilde{b}_1(0,0)-\tilde{b}_1(1,1)+\tilde{b}_1(0,1)+\tilde{b}_1(1,0)\right).
\end{equation}
Consequently, by (\ref{Gdef}), (\ref{G0for}) and (\ref{H1for3}),
\begin{equation*}
\lim_{\alpha\to\infty}\frac{\partial F_{tt}}{\partial\xi_1}(\xi_1,\eta_1;\frac{\xi+\alpha\xi_1}{\alpha'},\frac{\eta+\alpha^2\xi_1}{\alpha'^2};\alpha)
=F_2'(\xi_1+\eta_1^2)\frac{\partial}{\partial\xi}\left(F_2(\xi+\eta^2)\psi(\xi,\eta)\right),
\end{equation*}
where
\begin{equation*}
\psi(\xi,\eta)=\xi-\tilde{b}_1(0,0)-\tilde{b}_1(1,1)+\tilde{b}_1(0,1)+\tilde{b}_1(1,0).
\end{equation*}
This completes the proof of the theorem.
\end{proof}

\section{Proof of some Lemmas}\label{secLemmas}
In this section we will give the proofs of some Lemmas from the previous sections.

\begin{proof}[Proof of Lemma \ref{Lema0}]
From (\ref{Ak}) and (\ref{Ai}), we see that
\begin{equation}\label{A1diff}
A_{2,\pm}^{(1)}(v)=-\frac{d}{dv}A_{2,\pm}(v)=-\frac{\partial}{\partial \xi_2}A_{2,\pm}(v).
\end{equation}
Thus, by (\ref{Kistar}) and (\ref{Akernel}),
\begin{align*}
\frac{\partial}{\partial \xi_2}\bK_{2}(v_1,v_2)&=\frac{\partial}{\partial \xi_2}\int_{\bbR_+}A_{2,-}(v_1+v_3)A_{2,+}(v_3+v_2)\,dv_3\\&=
\int_{\bbR_+}\frac{\partial}{\partial v_3}\left(A_{2,-}(v_1+v_3)A_{2,+}(v_3+v_2)\right)\,dv_3
=-(A_{2,-}\otimes A_{2,+})(v_1,v_2).
\end{align*}
Consequently,
\begin{align*}
\frac{\partial}{\partial \xi_2}F_2(\xi_2+\eta_2^2)&=\frac{\partial}{\partial \xi_2}\det(\bI-\bK_2^*)=-\det(\bI-\bK_2^*)\Tr(\bI-\bK_2^*)^{-1}\frac{\partial\bK_2^*}{\partial \xi_2}\\
&=F_2(\xi_2+\eta_2^2)\Tr(\bI-\bK_2^*)^{-1}A_{2,-}\otimes A_{2,+}=F_2(\xi_2+\eta_2^2)a_0,
\end{align*}
by (\ref{a0}), and we have proved (\ref{a0TW}). Furthermore, by (\ref{A1diff}),
\begin{align*}
\frac{\partial a_0}{\partial \xi_2}&=-\int_{\bbR_+^2}A_{2,+}^{(1)}(v_1)(\bI-\bK_2^*)^{-1}(v_1,v_2)A_{2,-}(v_2)\,d^2v\\
&+\int_{\bbR_+^2}A_{2,+}(v_1)(\bI-\bK_2^*)^{-1}\frac{\partial\bK_2^*}{\partial \xi_2}(\bI-\bK_2^*)^{-1}(v_1,v_2)A_{2,-}(v_2)\,d^2v\\
&-\int_{\bbR_+^2}A_{2,+}(v_1)(\bI-\bK_2^*)^{-1}(v_1,v_2)A_{2,-}^{(1)}(v_2)\,d^2v=-a_1^*-a_0^2-a_1.
\end{align*}
If $\eta_1=0$, then $a_1^*=a_1$ by (\ref{a1a1star}), and hence 
\begin{equation*}
a_1=\frac 12\left(-a_0^2-\frac{\partial a_0}{\partial \xi_2}\right)=-\frac {F_2''(\xi_2)}{2F_2(\xi_2)},
\end{equation*}
where we used (\ref{a0TW}) with $\eta_2=0$.
\end{proof}

\begin{proof}[Proof of Lemma \ref{Lembrsrelation}]
Changing $\eta_1$ to $-\eta_1$ interchanges $A_{1,-}$ and $A_{1,+}$, and hence changes $\bK_1$ to $\bK_1^*$. Thus, by (\ref{brs}),
\begin{align*}
b_{r,s}(\epsilon_1,\epsilon_2)(\xi_1,-\eta_1)&=\int_{\bbR_+^2}(\bA_{1,+}^{\epsilon_1}(\bI-\bK_1^*)^{-r}\bA_{1,-}^{\epsilon_2})
(\lambda_1,\lambda_2)\frac{\lambda_2^s}{s!}\,d^2\lambda\\
&=\int_{\bbR_+^4}(\bA_{1,+}^{\epsilon_1}(\lambda_1,\lambda_3)(\bI-\bK_1^*)^{-r}(\lambda_3,\lambda_4)\bA_{1,-}^{\epsilon_2})
(\lambda_4,\lambda_2)\frac{\lambda_2^s}{s!}\,d^4\lambda.
\end{align*}
Interchanging $\lambda_1$ and $\lambda_2$, and $\lambda_3$ and $\lambda_4$, gives
\begin{align*}
b_{rs}(\epsilon_1,\epsilon_2)(\xi_1,-\eta_1)
&=\int_{\bbR_+^4}(\bA_{1,+}^{\epsilon_1}(\lambda_2,\lambda_4)(\bI-\bK_1^*)^{-r}(\lambda_4,\lambda_3)\bA_{1,-}^{\epsilon_2})
(\lambda_3,\lambda_1)\frac{\lambda_1^s}{s!}\,d^4\lambda\\
&=\int_{\bbR_+^4}\frac{\lambda_1^s}{s!}(\bA_{1,.}^{\epsilon_2}(\lambda_2,\lambda_4)(\bI-\bK_1^*)^{-r}(\lambda_3,\lambda_4)\bA_{1,+}^{\epsilon_1})
(\lambda_4,\lambda_2)\,d^4\lambda\\
&=b_{r,s}^*(\epsilon_2,\epsilon_1)(\xi_1,\eta_1),
\end{align*}
since $A_{1,-}$ and $A_{1,+}$ are symmetric. This proves (\ref{brsbrsstar}). 

Also, to prove (\ref{b0sb0sstar}), note that
\begin{equation*}
b_{0,s}(0,1)=\int_{\bbR_+^2}\bA_{1,+}(\lambda_1+\lambda_2)\frac{\lambda_2^s}{s!}\,d^2\lambda=
\int_{\bbR_+^2}\frac{\lambda_1^s}{s!}\bA_{1,+}(\lambda_1+\lambda_2)\,d^2\lambda=b_{0,s}^*(0,1),
\end{equation*}
and analogously for $b_{0,s}(1,0)$. The formula (\ref{br0br0star}) follows immediately from the definitions.

By (\ref{a1}),
\begin{equation*}
a_1(\xi_2,-\eta_2)=\Tr(\bI-\bK_2)^{-1}A_{2,+}^{(1)}\otimes A_{2,-}=
\int_{\bbR_+^2}\bA_{2,-}(\lambda_1)(\bI-\bK_2)^{-1}(\lambda_1,\lambda_2)A_{2,+}^{(1)}(\lambda_2)\,d^2\lambda.
\end{equation*}
Interchanging $\lambda_1$ and $\lambda_2$ shows that this equals,
\begin{equation*}
\int_{\bbR_+^2}\bA_{2,+}^{(1)}(\lambda_1)(\bI-\bK_2^*)^{-1}(\lambda_1,\lambda_2)A_{2,-}(\lambda_2)\,d^2\lambda=a_1(\xi_2,\eta_2),
\end{equation*}
since $\bK_2(\lambda_2,\lambda_1)=\bK_2^*(\lambda_1,\lambda_2)$. This proves (\ref{a1a1star}). The proof of (\ref{r1rel}) is analogous.
\end{proof}

\begin{proof}[Proof of Lemma \ref{LemkiMiformulas}]
We will only consider a few cases. The other are handled in a completely analogous way.

Note that since $\alpha'=1+O(\alpha^3)$ we can set $\alpha'=1$ when computing $\alpha$-derivatives up to order 2 at $\alpha=0$. If $f$ is a function of $\alpha$, $\Delta\xi$
and $\Delta\eta$, we see from (\ref{deltaxideltaeta}) that
\begin{align}\label{diff}
\left.\frac{df}{d\alpha}\right|_{\alpha=0}&=\left.\frac{\partial f}{\partial\alpha}\right|_{\alpha=0}-\xi_1\left.\frac{\partial f}{\partial\Delta\xi}\right|_{\alpha=0},\\
\left.\frac{d^2f}{d\alpha^2}\right|_{\alpha=0}&=\left.\frac{\partial^2 f}{\partial\alpha^2}\right|_{\alpha=0}-2\xi_1\left.\frac{\partial^2 f}{\partial\alpha\partial\Delta\xi}\right|_{\alpha=0}
-2\eta_1\left.\frac{\partial f}{\partial\Delta\eta}\right|_{\alpha=0}+\xi_1^2\left.\frac{\partial^2 f}{\partial\Delta\xi^2}\right|_{\alpha=0}.\notag
\end{align}
From (\ref{M3}), we see that
\begin{align*}
\bM_{3,0}(v_1,v_2)&=\frac{1}{(\pii)^2}\int_{\Gamma_D}dz\int_{\Gamma_{-d}}d\zeta\frac{G_{\xi_2+v_2,\eta_2}(z)}{G_{\xi_2+v_1,\eta_2}(\zeta)(z-\zeta)}\\
&=\int_0^\infty A_{2,-}(v_1,\lambda)A_{2,+}(\lambda,v_2)\,d\lambda=\bK_2^*(v_1,v_2).
\end{align*}
Using (\ref{M3}), (\ref{GAk1}), (\ref{GAk2}) and (\ref{diff}), we get
\begin{align*}
\bM_{3,1}(v_1,v_2)&=-\xi_1\frac{1}{(\pii)^2}\int_{\Gamma_D}dz\int_{\Gamma_{-d}}d\zeta\frac{G_{\xi_2+v_2,\eta_2}(z)(\zeta-z)}{G_{\xi_2+v_1,\eta_2}(\zeta)(z-\zeta)}\\
&=\xi_1\left(\frac 1{\pii}\int_{\Gamma_{-d}}\frac{d\zeta}{G_{\xi_2+v_1,\eta_2}(\zeta)}\right)\left(\int_{\Gamma_D}G_{\xi_2+v_2,\eta_2}(z)\,dz\right)=
\xi_1A_{2,-}\otimes A_{2,+}(v_1,v_2),
\end{align*}
and 
\begin{align*}
\bM_{3,2}(v_1,v_2)&=-2\eta_1\frac{1}{(\pii)^2}\int_{\Gamma_D}dz\int_{\Gamma_{-d}}d\zeta\frac{G_{\xi_2+v_2,\eta_2}(z)(z^2-\zeta^2)}{G_{\xi_2+v_1,\eta_2}(\zeta)(z-\zeta)}\\
&+\xi_1^2\frac{1}{(\pii)^2}\int_{\Gamma_D}dz\int_{\Gamma_{-d}}d\zeta\frac{G_{\xi_2+v_2,\eta_2}(z)(\zeta-z)^2}{G_{\xi_2+v_1,\eta_2}(\zeta)(z-\zeta)}\\
&=(\xi_1^2-2\eta_1)\left(\frac 1{\pii}\int_{\Gamma_{-d}}\frac{d\zeta}{G_{\xi_2+v_1,\eta_2}(\zeta)}\right)\left(\int_{\Gamma_D}zG_{\xi_2+v_2,\eta_2}(z)\,dz\right)\\
&+(-\xi_1^2-2\eta_1)\left(\frac 1{\pii}\int_{\Gamma_{-d}}\frac{\zeta d\zeta}{G_{\xi_2+v_1,\eta_2}(\zeta)}\right)\left(\int_{\Gamma_D}G_{\xi_2+v_2,\eta_2}(z)\,dz\right)\\
&=(\xi_1^2-2\eta_1)A_{2,-}\otimes A_{2,+}^{(1)}(v_1,v_2)+(\xi_1^2+2\eta_1)A_{2,-}^{(1)}\otimes A_{2,+}(v_1,v_2),
\end{align*}
by (\ref{GAk1}) and (\ref{GAk2}).

It follows immediately from (\ref{k1}) that $\bk_{1,0}=0$. Using (\ref{diff}), we see that
\begin{align*}
\bk_{1,1}(v_1,v_2)&=\frac{1}{(\pii)^4}\int_{\Gamma_{D_3}}dz\int_{\Gamma_{D_2}}dw\int_{\Gamma_{-d_3}}d\zeta\int_{\Gamma_{-d_2}}d\omega
\frac{G_{\xi_1,\eta_1}(z)G_{\xi_2+v_2,\eta_2}(w)}{G_{\xi_1,\eta_1}(\zeta)G_{\xi_2+v_1,\eta_2}(\omega)(z-\zeta)z(-\zeta)}\\
&=A_{2,-}(v_1)\bA_{1,+}^{(-1)}\bA_{1,-}^{(-1)}(0,0)A_{2,+}(v_2)=b_0(0,0)A_{2,-}\otimes A_{2,+}(v_1,v_2),
\end{align*}
since, using the definitions,
\begin{align*}
\bA_{1,+}^{(-1)}\bA_{1,-}^{(-1)}(0,0)&=\int_{\bbR_+}\bA_{1,+}^{(-1)}(0,\lambda_1)\bA_{1,-}^{(-1)}(,\lambda_1,0)\,d\lambda_1\\
&=\int_{\bbR_+^3}A_{1,+}(\lambda_1+\lambda_2)A_{1,-}(\lambda_1+\lambda_3)\,d^3\lambda=\int_{\bbR_+^2}\bK_1(\lambda_1,\lambda_2)\,d^2\lambda=b_0(0,0).
\end{align*}
From (\ref{k1}) and (\ref{diff}), it follows that
\begin{align*}
&\bk_{1,2}(v_1,v_2)=\left.\frac{\partial^2 \bk_1}{\partial\alpha^2}\right|_{\alpha=0}(v_1,v_2)-2\xi_1\left.\frac{\partial^2 \bk_1}{\partial\alpha\partial\Delta\xi}\right|_{\alpha=0}(v_1,v_2)\\
&=\frac{1}{(\pii)^4}\int_{\Gamma_{D_3}}dz\int_{\Gamma_{D_2}}dw\int_{\Gamma_{-d_3}}d\zeta\int_{\Gamma_{-d_2}}d\omega
\frac{G_{\xi_1,\eta_1}(z)G_{\xi_2+v_2,\eta_2}(w)}{G_{\xi_1,\eta_1}(\zeta)G_{\xi_2+v_1,\eta_2}(\omega)(z-\zeta)}\left[-\frac{2w}{z^2\zeta}-\frac{2\omega}{z\zeta^2}
+2\xi_1\frac{\omega-w}{z\zeta}\right]\\
&=2\bA_{1,+}^{(-2)}\bA_{1,-}^{(-1)}(0,0)A_{2,-}\otimes A_{2,+}^{(1)}(v_1,v_2)+2\bA_{1,+}^{(-1)}\bA_{1,-}^{(-2)}(0,0)A_{2,-}^{(1)}\otimes A_{2,+}(v_1,v_2)\\
&+2\xi_1\bA_{1,+}^{(-1)}\bA_{1,-}^{(-1)}(0,0)\left(A_{2,-}^{(1)}\otimes A_{2,+}+A_{2,-}\otimes A_{2,+}^{(1)}\right)(v_1,v_2).
\end{align*}
We now use,
\begin{equation*}
\bA_{1,+}^{(-2)}\bA_{1,-}^{(-1)}(0,0)=b_{0,1}^*(1,1),\quad \bA_{1,+}^{(-1)}\bA_{1,-}^{(-2)}(0,0)=b_{0,1}(1,1).
\end{equation*}
We can now proceed with the other cases in an analogous way. For these computations, we also use
\begin{align*}
A_{1,+}^{(-2)}(0)&=b_0(0,1),\quad A_{1,-}^{(-2)}(0)=b_0(1,0),\\
A_{1,+}^{(-3)}(0)&=2b_{0,1}(0,1),\quad A_{1,-}^{(-3)}(0)=b_{0,1}(1,0).
\end{align*}
\end{proof}

\begin{proof}[Proof of Lemma \ref{LemQ0exp}]
It follows from (\ref{Qr}) and Lemma \ref{LemkiMiformulas} that
\begin{equation}\label{IdQ0}
\bI+\bQ_0=
\begin{pmatrix} \bI-\bK_2^*  &  0\\
(1-u^{-1})\bk_{6,0}-\bk_{7,0} & \bI+(u^{-1}-1)\bM_1
\end{pmatrix}.
\end{equation}
The block-inverse formula
\begin{equation*}
\begin{pmatrix} A_{11} &  0\\
A_{21} & A_{22}
\end{pmatrix}^{-1}=
\begin{pmatrix} A_{11}^{-1} &  0\\
-A_{22}^{-1}A_{21}A_{11}^{-1} & A_{22}^{-1}
\end{pmatrix},
\end{equation*}
then gives
\begin{equation}\label{IdQ0inv}
(\bI+\bQ_0)^{-1}=
\begin{pmatrix} \bL_2^*  &  0\\
(\bI+(u^{-1}-1)\bM_0)^{-1}((u^{-1}-1)\bk_{6,0}-\bk_{7,0})\bL_2^* & (\bI+(u^{-1}-1)\bM_1)^{-1}
\end{pmatrix}
\end{equation}
if $|u|$ is sufficiently large. We can write
\begin{equation}\label{M1u}
\bI+(u^{-1}-1)\bM_1=(\bI-\bM_1)(\bI+u^{-1}(\bI-\bM_1)^{-1})=(\bI-\bM_1)(\bI+u^{-1}\bR).
\end{equation}
Thus, (\ref{IdQ0}) gives
\begin{equation*}
\det(\bI+\bQ_0)=\det(\bI-\bK_2^*)\det(\bI-\bM_1)\det(\bI+u^{-1}\bR)=
F_2(\xi_1+\eta_1^2)F_2(\xi_2+\eta_2^2)\det(\bI+u^{-1}\bR).
\end{equation*}
This proves (\ref{detQ0for}). The expansion (\ref{detRexp}) follows from
\begin{equation*}
\log\det(\bI+u^{-1}\bR)=\sum_{\ell=1}^\infty\frac{(-1)^{\ell-1}}{\ell!}u^{-k}\Tr \bR^k,
\end{equation*}
if $|u|$ is sufficiently large. From (\ref{M1u}), we see that
\begin{equation*}
(\bI+(u^{-1}-1)\bM_1)^{-1}=(\bI+u^{-1}\bR)^{-1}(\bI-\bM_1)^{-1},
\end{equation*}
and using the fact that
\begin{equation*}
(\bI+u^{-1}\bR)^{-1}=\sum_{k=0}^\infty(-1)^ku^{-k}\bR^k,
\end{equation*}
for $|u|$ is sufficiently large, we get
\begin{equation*}
(\bI+(u^{-1}-1)\bM_1)^{-1}=\sum_{k=0}^\infty(-1)^ku^{-k}\bR^k(\bI-\bM_1)^{-1}.
\end{equation*}
Inserting this into (\ref{IdQ0inv}) gives
\begin{equation*}
(\bI+\bQ_0)^{-1}=
\begin{pmatrix} \bL_2^*  &  0\\
\sum_{k=0}^\infty(-1)^ku^{-k}\bR^k\tilde{\bL}_1((u^{-1}-1)\bk_{6,0}-\bk_{7,0})\bL_2^* & \sum_{k=0}^\infty(-1)^ku^{-k}\bR^k\tilde{\bL}_1
\end{pmatrix},
\end{equation*}
from which (\ref{Q0invexp}), (\ref{Q00}) and (\ref{Q0k}) follow.
\end{proof}

\begin{proof}[Proof of Lemma \ref{Leme1e2formula}]
We see from (\ref{Fttformula1}), (\ref{detQexp}) and (\ref{Prkdef}) that
\begin{equation}\label{Fttexp}
F_{tt}(\xi_1,\eta_1;\xi_2,\eta_2;\alpha)=\frac 1{\pii}\int_{\gamma_r}\det(\bI+\bQ_0)\left[(\Tr\bP_1)\alpha+
\frac 12\left((\Tr\bP_1)^2-\Tr\bP_1^2+\Tr\bP_2\right)\alpha^2\right]\frac{du}{u-1}+O(\alpha^3).
\end{equation}
From this formula, (\ref{detQ0for}), (\ref{detRexp}) and (\ref{Prkdef}), we obtain
\begin{equation}\label{e1comp+}
e_1=\frac 1{\pii}\int_{\gamma_r}\left(1+\sum_{k=1}^\infty r_ku^{-k}\right)\left(\sum_{k=-1}^\infty\Tr\bP_1(k)u^{-k}\right)\frac{du}{u-1},
\end{equation}
and
\begin{align}\label{e2comp}
e_2&=\frac 1{\pii}\int_{\gamma_r}\frac 12\left(1+\sum_{k=1}^\infty r_ku^{-k}\right)\left[\left(\sum_{k=-1}^\infty\Tr\bP_1(k)u^{-k}\right)\left(\sum_{\ell=-1}^\infty\Tr\bP_1(\ell)u^{-\ell}\right)
\right.\\&-\left.\Tr\left(\sum_{k=-1}^\infty\bP_1(k)u^{-k}\right)\left(\sum_{\ell=-1}^\infty\bP_1(\ell)u^{-\ell}\right)+\sum_{k=-1}^\infty\Tr\bP_2(k)u^{-k}\right]\frac{du}{u-1}.\notag
\end{align}
We now use the fact that
\begin{equation}\label{uint1}
\frac 1{\pii}\int_{\gamma_r}\frac{1}{u^k(u-1)}du=\begin{cases} 0 & \text{if $k\ge 1$} \\
1 & \text{if $k\le 0$},
\end{cases}
\end{equation}
for all $r>1$. Using (\ref{e1comp+}) this gives (\ref{e1formula}), and from (\ref{e2comp}), we get (\ref{e2formula}) by (\ref{sigmakl}).
\end{proof}

\begin{proof}[Proof of Lemma \ref{Lemgrsformula}]
We have that
\begin{align*}
&g_{r,s}(\epsilon_1,\epsilon_2)\\
&=\int_{\bbR_+^2}\left(\int_{\bbR_+}\bA_{1,-}^{\epsilon_1}(\lambda_3,\lambda_1)\,d\lambda_3\right)
(\bI-\bK_1)^{-r}\bK_1^{r-1}(\lambda_1,\lambda_2)\left(\int_{\bbR_+}(\bK_1\bA_{1,+}^{\epsilon_2})(\lambda_2,\lambda_4)\lambda_4^s\,d\lambda_4\right)
\,d\lambda_1d\lambda_2\\
&=\int_{\bbR_+^2}(\bA_{1,-}^{\epsilon_1}(\bI-\bK_1)^{-r}\bK_1^{r-1}\bA_{1,+}^{\epsilon_2})(\lambda_3,\lambda_4)\frac{\lambda_4^s}{s!}\,d\lambda_3d\lambda_4.
\end{align*}
From the identity,
\begin{equation*}
(\bI-\bK_1)^{-1}\bK_1=(\bI-\bK_1)^{-1}-\bI,
\end{equation*}
we get
\begin{equation*}
(\bI-\bK_1)^{-r}\bK_1^r=\sum_{k=0}^r(-1)^{r-k}\binom{n}{k}(\bI-\bK_1)^{-k}.
\end{equation*}
Thus
\begin{align*}
g_{r,s}(\epsilon_1,\epsilon_2)&=\sum_{k=0}^r(-1)^{r-k}\binom{n}{k}\int_{\bbR_+^2}(\bA_{1,-}^{\epsilon_1}(\bI-\bK_1)^{-k}
\bA_{1,+}^{\epsilon_2})(\lambda_1,\lambda_2)\frac{\lambda_2^s}{s!}\,d^2\lambda\\
&=\sum_{k=0}^r(-1)^{r-k}\binom{n}{k}b_{k,s}(\epsilon_1,\epsilon_2).
\end{align*}
Similarly,
\begin{align*}
g_{rs}(0,0)&=\int_{\bbR_+^2}(\bI-\bK_1)^{-r}\bK_1^{r-1}(\lambda_1,\lambda_2)\bK_1^{(-1-s)}(\lambda_2)\,d^2\lambda\\
&=\int_{\bbR_+^2}(\bI-\bK_1)^{-r}\bK_1^{r}(\lambda_1,\lambda_2)\frac{\lambda_2^s}{s!}\,d^2\lambda\\
&=\int_{\bbR_+^2}(\bI-\bK_1)^{-1}\bK_1\left[\sum_{k=0}^{r-1}(-1)^{r-1-k}\binom{r-1}{k}(\bI-\bK_1)^{-k}\right](\lambda_1,\lambda_2)\frac{\lambda_2^s}{s!}\,d^2\lambda\\
&=\sum_{k=1}^r(-1)^{r-k}\binom{r-1}{k-1}\int_{\bbR_+^2}(\bI-\bK_1)^{-k}\bK_1(\lambda_1,\lambda_2)\frac{\lambda_2^s}{s!}\,d^2\lambda\\
&=\sum_{k=1}^r(-1)^{r-k}\binom{r-1}{k-1}b_{k,s}(0,0).
\end{align*}
\end{proof}

\begin{proof}[Proof of Lemma \ref{LemQF2}]
It follows from (\ref{Qformula1}) that
\begin{align}\label{detM2}
\det(\bI+\bQ(1,\beta,\Delta\xi,\xi_1,\Delta\eta,\eta_1,\delta))_Y&=
\det\begin{pmatrix} \bI-\bM_2(\beta,\Delta\xi,\xi_1,\Delta\eta,\eta_1,\delta) & 0 \\
*  &  \bI
\end{pmatrix}_Y\\
&=\det(\bI-\bM_2(\beta,\Delta\xi,\xi_1,\Delta\eta,\eta_1,\delta))_{L^2(\bbR_+)}.\notag
\end{align}
Now, using (\ref{M2}), we see that
\begin{equation}\label{M2new}
\bM_2(\beta,\Delta\xi,\xi_1,\Delta\eta,\eta_1,\delta)(v_1,v_2)=
\frac{1}{(\pii)^2\beta'}\int_{\Gamma_D}dz\int_{\Gamma_{-d}}d\zeta\frac{G_{\xi_2+v_2/\beta',\eta_2}(z)}{G_{\xi_2+v_1/\beta',\eta_2}(\zeta)(z-\zeta)}.
\end{equation}
Recall that $\xi_2$ in (\ref{M2}) is given by $\xi_2=\xi_2(\alpha,\xi_1,\Delta\xi)=(\Delta\xi+\alpha\xi_1)/\alpha'$, so in (\ref{M2new}), we have instead
\begin{equation*}
\xi_2(\beta,\Delta\xi,\xi_1)=\frac{\xi_1+\beta\Delta\xi}{\beta'}=\frac{\xi_1+\beta(\alpha'\xi_2-\alpha\xi_1)}{\beta'}=\xi_2,
\end{equation*}
so, in fact, we just get $\xi_2$ in (\ref{M2new}). An analogous argument applies to $\eta_2$. From (\ref{M2new}), we obtain
\begin{equation*}
\bM_2(\beta,\Delta\xi,\xi_1,\Delta\eta,\eta_1,\delta)(v_1,v_2)=\frac 1{\beta'}\bK_2^*(\frac{v_1}{\beta'},\frac{v_2}{\beta'}),
\end{equation*}
from which we see that the right side of (\ref{detM2}) is $F_2(\xi_2+\eta_2^2)$.
\end{proof}


\end{document}